\documentclass{amsart}
\usepackage{amsmath}
\usepackage{amssymb}
\usepackage[foot]{amsaddr}
  \usepackage{paralist}
  \usepackage{graphics} 
  \usepackage{epsfig} 
\usepackage{graphicx}  \usepackage{epstopdf}
 \usepackage[colorlinks=true]{hyperref}
\hypersetup{urlcolor=blue, citecolor=red}
\usepackage{bm}
\usepackage{verbatim}
\usepackage{mathtools}

\numberwithin{equation}{section}

\newtheorem{theorem}{Theorem}[section]
\newtheorem{corollary}[theorem]{Corollary}

\newtheorem{lemma}[theorem]{Lemma}
\newtheorem{proposition}[theorem]{Proposition}

\theoremstyle{definition}
\newtheorem{definition}[theorem]{Definition}

\newtheorem{example}[theorem]{Example} 

\theoremstyle{remark}

\newtheorem{claim}{Claim}

\DeclareMathOperator{\Homeo}{Homeo}

\DeclareMathOperator{\Ph}{Ph}
\DeclareMathOperator{\im}{im}
\DeclareMathOperator{\dom}{dom}
\DeclareMathOperator{\id}{id}
\DeclareMathOperator{\diam}{diam}
\DeclareMathOperator{\lcm}{lcm}
\DeclareMathOperator{\Eq}{Eq}

\newcommand{\acts}{\curvearrowright}

\usepackage{caption}

\begin{document}

\title{Chaos for foliated spaces and pseudogroups}
\author{Ram\'on Barral Lij\'o}
\address{E.T.S.~de Ingenieros Inform\'aticos, Universidad Polit\'ecnica de Madrid, 28660 Madrid, Espa\~na}
\email{ramon.barral@upm.es}
\thanks{The author is partially supported by grants PID2020-114474GB-I00 (AEI/FEDER, UE), ED431C 2019/10 (Xunta de Galicia, FEDER), and by the Research Promotion Program for Acquiring Grants in-Aid for Scientific Research (KAKENHI) by Ritsumeikan University.}
\thanks{MSC 2020: 37B02, 37B05.}
\thanks{Keywords: chaos, sensitivity, foliation, foliated space, Devaney chaos, equicontinuity, pseudogroup, periodic orbits, lamination}
\thanks{Declarations of interest: none}

\begin{abstract}
	We generalize  ``sensitivity to initial conditions'' to foliated spaces and pseudogroups, offering  a  definition of Devaney chaos in this setting. In contrast to the case of group actions, where sensitivity follows from the other two conditions of Devaney chaos, we show that this is true only for compact foliated spaces, exhibiting a counterexample in the non-compact case. Finally, we obtain an analogue of the Auslander--Yorke dichotomy for compact foliated spaces and compactly generated pseudogroups.
\end{abstract}

\maketitle

\section{Introduction}

There are several definitions of chaos for dynamical systems (Li--Yorke chaos, positive entropy, ...), but in this article we will consider only Devaney's, first introduced in~\cite{Devaney}.

\begin{definition}[Devaney chaos] \label{d:dev}
	A continuous map $f\colon X\to X$ on a metric space $(X,d)$ is \emph{chaotic} if,
	\begin{enumerate}[(i)]
		\item \label{i:devtt}
		for all non-empty open $U,V\subset X$, there is $n\geq 0$ such that
		\[
		f^n(U)\cap V\neq \emptyset
		\]
		($f$ is \emph{topologically transitive}),
		\item \label{i:devper}the set of periodic points is dense in $X$  ($f$ has \emph{density of periodic points}),	and	
		\item \label{i:devsens} there is $c>0$ such that, for every $x\in X$ and $r>0$, there are $y\in B(x,r)$ and $n\geq 0$ satisfying 
		\[
		d(f^n(x),f^n(y))\geq c
		\] ($f$ is \emph{sensitive to initial conditions}).
	\end{enumerate}
\end{definition}

This definition can be readily adapted for group actions $G\acts X$ by substituting $g\in G$ in place of $f^n$ ($n\in \mathbb{N}$\footnote{We will adhere to the convention where $0\in\mathbb{N}$.}) above.
Topological transitivity conveys the indecomposability of the dynamical system, whereas~(\ref{i:devper}), according to Devaney himself, provides ``an element of regularity''~\cite[p.~50]{Devaney}. Sensitivity to initial conditions, for its part, expresses what is commonly known as the ``butterfly effect''. This rough sketch may lead to the impression that~(\ref{i:devsens}) alone imbues this definition with its chaotic nature; suprisingly, it was proved later that this condition is, in fact, redundant.
\begin{theorem}[{\cite{BBCDS}}]\label{t:bbcds}
	If a continuous map $f\colon X\to X$ on a metric space $(X,d)$ satisfies~(\ref{i:devtt}) and~(\ref{i:devper}), then it also satisfies~(\ref{i:devsens}).
\end{theorem}
This result was later generalized to topological group and semigroup actions~\cite{Kontorovich,SKBS}. The reader should bear in mind that these results hold even when the phase space is not compact.

If the local behaviour of $f$ around a point $x$ is not sensitive to initial conditions, then there is an assignment $\epsilon\mapsto\delta(\epsilon)$ such that 
\[
d(x,y)<\delta(\epsilon)\quad\Longrightarrow\quad d(f^nx,f^ny)<\epsilon \qquad\text{for every}\quad y\in X,\ n\in\mathbb{N},
\]
and we say that $x$ is a \emph{point of equicontinuity}. If the set of points of equicontinuity is dense in $X$, we call the dynamical system  \emph{almost equicontinuous}; if every point is of equicontinuity with the same modulus $\epsilon\mapsto \delta(\epsilon)$, then the system is \emph{equicontinuous}. This rough opposition between chaos and equicontinuity is rigorously formulated
 by the Auslander--Yorke dichotomy.
	
	\begin{theorem}[{Auslander--Yorke dichotomy~\cite[Cor.~2]{AuslanderYorke}}]\label{t:auslanderyorkemin}
		Let $X$ be a compact space and let $f\colon X\to X$ be a continuous map such that $(X,f)$ is minimal. Then $f$ is either equicontinuous or sensitive to initial conditions.
	\end{theorem} 
	
After this brief review, we can state the aim of the present paper: to study topological chaos for foliations and their generalization, foliated spaces; as we will see, this requires considering pseudogroups too. Recall that a foliated space is a topological generalization of a foliation where the choice of local transversal models is not restricted to manifolds: they are only required to be Polish spaces (see Section~\ref{ss:foliated}). The Smale--Williams attractor provides an example of a foliated space that is not a foliation---it is locally homeomorphic to the product of the real line and the Cantor set.

Thus, our work fits into the broader field of topological dynamics for foliated spaces, which has received much attention of late. The most studied foliation dynamics are the equicontinuous, featuring the celebrated tools of Molino theory for \emph{Riemannian foliations}~\cite{Molino}. Equicontinuity was generalized to foliated spaces and pseudogroups in~\cite{Sacksteder, Molino, Kellum, Tarquini, AlvarezCandel} with varying degrees of generality, the methods of~\cite{AlvarezCandel} in particular being a main source of inspiration for this paper. Molino theory itself has also been generalized in~\cite{AlvarezMoreira, DyerHurderLukina2016,DyerHurderLukina}, giving rise to the study of \emph{wild} solenoids and Cantor actions, which have a complicated interplay between local and global behavior~\cite{HurderLukina,HurderLukina2021,Lukina, AlvarezBarralLukinaNozawa}. There has been some recent work on complex dynamics, concerning Fatou--Julia decompositions for holomorphic foliations~\cite{GhysGomezSaludes,
	Asuke2013}.

In order to analyze foliated spaces from a dynamical point of view, we regard them as generalized dynamical systems where the leaves play the role of the orbits; as the title of~\cite{Churchill} reads, they are dynamical systems ``in the absence of time.'' We may identify the paper on  topological entropy for foliations by Ghys, Langevin, and Walczak~\cite{GLW} as the first study on chaotic foliations, even though the word ``chaos'' is never mentioned. In fact, it looks like the term ``chaotic foliation'' has only appeared twice in the literature; its debut was in the context of general relativity, where Churchill was trying to provide a definition of chaos invariant by relativistic reparametrizations of time:

\begin{definition}[{Churchill chaos~\cite{Churchill}}]
	A foliation is chaotic if
	\begin{enumerate}[(i)]
		\item \label{i:churchilltt}there is a dense leaf,
		\item \label{i:churchilldpp}the set of compact leaves is dense, and
		\item \label{i:churchillnontrivial}there are at least two different leaves.
	\end{enumerate}
\end{definition}

Items~(\ref{i:churchilltt}) and~(\ref{i:churchilldpp}) correspond to~\ref{d:dev}(\ref{i:devtt})--(\ref{i:devper}), whereas~(\ref{i:churchillnontrivial}) avoids the trivial scenario where the foliation consists of a single leaf.
Churchill did not include a foliation-theoretic definition of sensitivity: only foliations by curves arising from a flow were considered.

More recently, Bazaikin, Galaev, and Zhukova have provided the following definition of chaos for foliations:
\begin{definition}[Bazaikin--Galaev--Zhukova chaos~\cite{BGZ}]
	A foliation is chaotic if
	\begin{enumerate}[(i)]
		\item there is a dense leaf and
		\item the set of closed leaves is dense.
	\end{enumerate}
\end{definition}
They have used this definition to study chaos for Cartan foliations, relating it to conditions in their holonomy pseudogroups and global holonomy groups. Of course, their definition coincides with Churchill's when the ambient manifold is compact and there are at least two different leaves. Again, it does not take into account sensitivity to initial conditions. 

Regarding examples of chaotic foliated spaces (according to the definitions above), besides those appearing in~\cite{Churchill,BGZ}, the author was involved in the recent study of a hyperbolic version of the cut-and-project method of tiling theory~\cite{ABHNP}. This yields Delone subsets of $\mathbb{R}$ whose continuous hulls, which are naturally foliated spaces, are chaotic with respect to the natural action by translations.

This discussion motivates the first contribution of the present paper: we introduce a suitable definition of sensitivity to initial conditions for foliated spaces. We phrase this definition in terms of \emph{holonomy pseudogroups}, which have long been used as dynamical models for foliations. A pseudogroup in a topological space $X$ is a collection of homeomorphisms between open subsets of $X$ containing the identity and closed under composition, inversion, restriction, and combination of partial maps (see Section~\ref{s:pseudogroup}). H.~Nozawa and the author have developed a slightly different dynamical model in order to define sensitivity and Devaney chaos for closed saturated subsets of the Gromov space of pointed colored graphs~\cite{BarralNozawa}. These subsets resemble singular foliations by graphs and do not admit a holonomy pseudogroup in the usual sense. On a very related note, Flores and M\v{a}ntoiu have recently studied the topological dynamics of groupoid actions.

After some preliminary results in Section~\ref{s:preliminaries}, we discuss our definition of sensitivity for pseudogrous in Section~\ref{ss:sensitivityandchaos}:  showing first why a naive approach fails, we follow the ideas present in~\cite{AlvarezCandel} in order to arrive at Definition~\ref{d:sensitivity}.
We also provide definitions for almost equicontinuity and density of periodic orbits, making use of the latter to define Devaney chaos as follows:
\begin{definition}[Devaney chaos for pseudogroups]
	A pseudogroup $\mathcal G\acts X$ is \emph{chaotic} if it is topologically transitive, has density of periodic orbits, and is sensitive to initial conditions.
\end{definition}

 The next results test whether this new definition of sensitivity constitutes a satisfactory generalization of the original one. We start by examining pseudogroups generated by group actions:
\begin{theorem}\label{t:sicactionpseudogroup}
	If $G$ is a finitely generated group acting on a compact Polish space $X$, then the action is sensitive to initial conditions if and only if the pseudogroup generated by the action is.
\end{theorem} 
We also show in Section~\ref{s:nonsensitiveaction} that the conditions on $G$ and $X$ are  necessary for the result to hold. So, in general, sensitivity of the pseudogroup induced by a group action is strictly stronger than sensitivity of the group action itself.

\begin{theorem}\label{t:linked}
	There are group actions $G\acts X$ that are sensitive to initial conditions but such that the pseudogroup generated by the action is not, where either
	\begin{itemize}
		\item $G$ is the free group on two generators and $X=\mathbb{T}^2\times \mathbb{Z}$, where $\mathbb{T}^2$ is the $2$-torus, or
		\item $G$ is the free group on countably many generators and $X=\mathbb{T}^2$.
	\end{itemize}
\end{theorem} 

These actions are constructed using \emph{linked twists}, a family of classical examples of chaotic dynamical systems (see Section~\ref{s:linked} and the references therein). We will later recycle these counterexamples in the proof of Theorem~\ref{t:counterfol}.

We continue in Section~\ref{s:main} with our three main contributions regarding pseudogroup dynamics. Our first result addresses the following issue: if we are to use pseudogroups as dynamical models for foliated spaces, all our new definitions must be invariant by (Haefliger) equivalences. This is because the holonomy pseudogroup of a foliated space is only well-defined up to equivalence (see Sections~\ref{ss:equivalences} and~\ref{ss:foliated}). 

\begin{theorem}\label{t:invariant}
	Sensitivity to initial conditions, density of periodic orbits, Devaney chaos, and almost equicontinuity are invariant by equivalences of pseudogroups acting on locally compact Polish spaces.
\end{theorem}

The corresponding result for equicontinuity  was proved in~\cite{AlvarezCandel}. The main difficulty in~\ref{t:invariant} is proving the invariance of sensitivity. The reason why this is not trivial is that sensitivity and almost equicontinuity involve a metric, which is a global object; equivalences, however, are made up of local homeomorphisms, so we have to put in some work to construct a global metric using the information carried over by the local maps. 

Our next objective will be to study whether Theorems~\ref{t:bbcds} and~\ref{t:auslanderyorkemin} extend to the pseudogroup setting. 
We manage to do so  for compactly generated pseudogroups (see Section~\ref{ss:compactlygenerated} for the definition of compact generation).

\begin{theorem}[Auslander--Yorke dichotomy for pseudogroups]\label{t:auslanderyorkepseudo}
	Let $\mathcal{G}$ be a compactly generated and topologically transitive pseudogroup acting on a Polish space. Then $\mathcal{G}$ is either sensitive to initial conditions or almost equicontinuous. Moreover, if  $\mathcal{G}$ is  minimal, then it is either sensitive to initial conditions or equicontinuous.
\end{theorem}

\begin{theorem}\label{t:ttdposicpseudo}
	If $\mathcal{G}$ is a compactly generated and topologically transitive pseudogroup acting on a Polish space which  has density of periodic orbits, then it is sensitive to initial conditions.
\end{theorem}

Even though Theorem~\ref{t:bbcds} holds for actions on non-compact spaces, we exhibit in Section~\ref{s:cantor} a non-compactly generated, countably generated pseudogroup that is topologically transitive and has density of periodic orbits, but it is not sensitive to initial conditions. This shows that compact generation cannot be dropped in Theorem~\ref{t:ttdposicpseudo}.

At this point, we turn our attention to studying chaos for foliated spaces. By virtue of Theorem~\ref{t:invariant}, we can define almost equicontinuity and sensitivity using the holonomy pseudogroup.
\begin{definition}
	A foliated space is sensitive to initial conditions or almost equicontinuous if its holonomy pseudogroup is.
\end{definition}
Regarding density of periodic orbits and Devaney chaos, we encounter an additional subtlety: It is easy to check that density of periodic orbits for the holonomy pseudogroup implies density of closed leaves, but one might also consider the stronger condition of density of compact leaves. We choose the latter option because the counterexample we exhibit in Theorem~\ref{t:counterfol} satisfies this stronger condition. On the other hand, we run into the problem that density of compact leaves cannot be formulated as a equivalence-invariant property of the holonomy pseudogroup (see Example~\ref{e:rz}). 

\begin{definition}\label{d:chaosfs}
	A foliated space is \emph{chaotic} if it is topologically transitive, it has a dense set of compact leaves, and it is sensitive to initial conditions. 
\end{definition}

Note that, by the previous discussion, chaoticity of the foliated space is strictly stronger than chaoticity of the holonomy pseudogroup. We show in Section~\ref{s:chaoschaos} an explicit example of this behavior.

Our next step is to extend Theorems~\ref{t:bbcds} and~\ref{t:auslanderyorkemin} for compact foliated spaces, where density of compact leaves and of closed leaves coincide. By the previous discussion, these results follows immediately from Theorems~\ref{t:ttdposicpseudo} and~\ref{t:auslanderyorkepseudo}.
\begin{theorem}\label{t:ttdposicfol}
	Let $X$ be a compact Polish foliated space. If $X$ is topologically transitive and has density of compact leaves, then it is sensitive to initial conditions.
\end{theorem}

\begin{theorem}
	Let $X$ be a compact and topologically transitive Polish foliated space. Then $X$ is either sensitive to initial conditions or almost equicontinuous. Moreover, if $X$ is minimal, then it is either sensitive to initial conditions or equicontinuous.
\end{theorem}

In analogy to the case of pseudogroups, where we need compact generation in Theorem~\ref{t:ttdposicpseudo}, compactness cannot be dropped in Theorem~\ref{t:ttdposicfol}; a simple counterexample with totally disconnected transversals is constructed in Section~\ref{s:ttdposicfol}.
One might wonder whether it is possible to find similar counterexamples among non-compact foliations, perhaps with smooth transversal dynamics.
We conclude the paper  in Sections~\ref{s:affinepseudogroup} and~\ref{s:ttdclnotsicaffine} with the following counterexample: 

\begin{theorem}\label{t:counterfol}
	There is a foliation by surfaces on a smooth $4$-manifold that is topologically transitive and has a dense set of compact leaves, but is not sensitive to initial conditions. This foliation is $C^\infty$ and transversally affine.
\end{theorem}

We can offer the following geometrical interpretation of the lack of sensitivity in Theorem~\ref{t:counterfol}. For this non-compact, smooth $4$-manifold $M$, there is a locally finite foliated atlas $(U_i,\phi_i)$---where $\phi_i\colon U_i\to \mathbb{R}^2\times T_i$ and $T_i\subset \mathbb{R}^2$ are the local transversals---satisfying the following condition: every holonomy transformation is an affine map, and there is a leaf $L$ such that every holonomy transformation $h$ between transversals $T_i\to T_j$ induced by a path in $L$ is an isometry with respect to the Euclidean metric on $T_i, T_j\subset \mathbb{R}^2$.

The results of this paper confirm that our definition of chaos is the right one if we restrict our attention to compactly generated pseudogroups and compact foliated spaces. However, as soon as we drop compactness, it seems to become a strong condition, at least when compared to the case of group actions. Even though it is invariant by equivalences and it is phrased in a way that mirrors other pseudogroup dynamical properties in the literature, it is an open question whether one can find a definition better suited to the non-compact case. The examples in Sections~\ref{s:cantor} and~\ref{s:foliated} suggest that perhaps one should  require only a meager set of equicontinuity points. The author ignores whether this condition could extend the notion of sensitivity to actions of countably generated pseudogroups satisfactorily.

\section{Preliminaries}\label{s:preliminaries}

\subsection{Metric spaces}
In this paper,  we consider metric functions $d\colon X\times X\to [0,\infty]$ that may attain an infinite value. A metric on a topological space is said to be \emph{compatible} if the underlying topology agrees with that generated by the open balls. A topological space is \emph{Polish} if it is separable and it admits a compatible complete metric; that is, one where every Cauchy sequence converges. All topological spaces will be implicitly  assumed to be Polish.

A \emph{shrinking} of an indexed open covering $\{U_\alpha\}_{\alpha\in A}$ of some topological space is a covering $\{V_\alpha\}_{\alpha\in A}$ with the same index set and such that $\overline{V_\alpha}\subset U_\alpha$ for every $\alpha\in A$. A covering $\{U_\alpha\}_{\alpha\in A}$ of $X$ is \emph{locally finite} if every point $x\in X$ has an open neighborhood $W$ that intersects only finitely many of the sets $U_\alpha$. We will make use of the following result (see e.g.~\cite[p.~227]{Munkres}):

\begin{lemma}[Shrinking lemma]
    Let $X$ be a Polish space. If $\{U_n\}_{n\in\mathbb{N}}$ is a locally finite countable open covering of $X$, then $\{U_n\}_{n\in\mathbb{N}}$ admits a shrinking.
\end{lemma}

\subsection{Partial maps and pseudogroups}\label{s:pseudogroup}
Let $X$ and $Y$ be topological spaces. A \emph{partial map} from $X$ to $Y$ is a map $f\colon A\to Y$ with domain a subset $A\subset X$. Given a partial map $f$, let $\dom f$ and $\im f$ denote the domain and image of $f$, respectively. We say that a partial map $f$ from $X$ to $Y$ is a \emph{partial homeomorphism}  if $\dom f\subset X$ and $\im f\subset Y$ are open  and $f\colon\dom f\to \im f$ is a homeomorphism; we  denote by $\Ph(X,Y)$ the set of partial homeomorphisms from $X$ to $Y$. From now on, we use  $f(A)$ as shorthand for  $f(A\cap\dom f)$, where $f\in\Ph(X,Y)$ and $A\subset X$.

Given $f\in \Ph(X,Y)$ and $g\in\Ph(Y,Z)$, the \emph{composition} $gf\in \Ph(X,Z)$ is defined by
\[
\dom gf= f^{-1}(\dom g),\qquad (gf)(x)=g(f(x)).
\]
Given $f\in \Ph(X,Y)$ and an open set $U\subset \dom f$, the \emph{restriction} $f|_{U}$ has domain $U$ and image $f(U)$. 

Let $\{f_i\mid i\in I\}$ be a family of maps in $\Ph(X,Y)$ and suppose that
\begin{equation}\label{fifj}
(f_{i})|_{\dom f_i\cap \dom f_j}=(f_{j})|_{\dom f_i\cap \dom f_j}\qquad \text{for every}\ i,j\in I,
\end{equation}
then the \emph{combination} $\bigcup_{i\in I}f_i$ is defined by
\[
\dom \big(\bigcup_{i\in I}f_i\big) =\bigcup_{i\in I}(\dom f_i),\qquad \big(\bigcup_{i\in I}f_i \big)(x)= f_i(x) \quad \text{for} \ x\in \dom f_i.
\]
For $f,g\in\Ph(X,Y)$, we say that $f$ \emph{extends} $g$, or $f$ is an \emph{extension} of $g$, if 
\[
\dom g\subset \dom f\qquad  \text{and} \qquad f|_{\dom g}=g.\]
For brevity, we use $\Ph(X)$ to denote the set $\Ph(X,X)$.

\begin{definition}[\cite{CY}]\label{d:pseudogroup}
	A subset $\mathcal{G}\subset \Ph(X)$ is a \emph{pseudogroup} if the following conditions are satisfied:
	\begin{itemize}[--]
		\item Group-like axioms:
		\begin{enumerate}[(i)]
			\item \label{i:identity} $\id_X\in\mathcal{G}$,
			\item if $f\in\mathcal{G}$, then $f^{-1}\in \mathcal{G}$ (\emph{closure under inversion}), and
			\item if $f,g\in \mathcal{G}$, then $fg\in\mathcal{G}$ (\emph{closure under composition}).
		\end{enumerate}
		\item Sheaf-like axioms:
		\begin{enumerate}[(i)]
			\setcounter{enumi}{3}
			\item \label{i:restrictions} if $f\in \mathcal{G}$ and $U\subset \dom f$ is open, then $f|_{U}\in \mathcal{G}$ (\emph{closure under restrictions}), and,
			\item \label{i:combination} if $\{f_i,\ i\in I\}$ is a family of maps in $\mathcal{G}$ satisfying~\eqref{fifj}, then $\bigcup_{i\in I} f_i\in \mathcal{G}$.  (\emph{closure under combinations}).
		\end{enumerate}
	\end{itemize} 
The last axiom can be refomulated as follows:
\begin{enumerate}
	\item[(v)$'$] if $f\in \Ph(X)$ is such that every $x\in \dom f$ has some open neighborhood $U_x$  with $f|_{U_x}\in \mathcal{G}$, then $f\in \mathcal{G}$.
\end{enumerate}
\end{definition}

If $\mathcal{G}\subset\Ph(X)$ is a pseudogroup, we say that $\mathcal{G}$ \emph{acts} on $X$ and we denote it by $\mathcal{G}\acts X$.
The $\mathcal{G}$-orbit of a point $x\in X$ is the subset $\mathcal{G}x = \{gx \mid g\in\mathcal{G}\}$.
If $\{\mathcal{G}_i\}_{i\in I}$ is a collection of pseudogroups acting on $X$, then $\bigcap_{i\in I}\mathcal{G}_i\subset \Ph(X)$ is also a pseudogroup. A subset $S\subset\mathcal{G}$ \emph{generates} $\mathcal{G}$ if $\mathcal{G}$ is the smallest pseudogroup containing $S$; equivalently, $\mathcal{G}$ is the intersection of all the pseudogroups in $\Ph(X)$ that contain $S$.
Let $\mathcal{G}$ be a pseudogroup acting on $X$, and let $U$ be an open subset of $X$, then 
the \emph{restriction} 
\[
\mathcal{G|}_{U}=\{\,f\in\mathcal{G} \mid \dom f\subset U,\ \im f\subset U \,\}
\]
is a pseudogroup acting on $U$.

One can find in the literature definitions of pseudogroup that omit Axiom~(\ref{i:combination}) (e.g., \cite{HectorHirsch}). The reason is that, by allowing combinations, a pseudogroup  might have ``too many maps'' to satisfy reasonable dynamical properties
 (see Definition~\ref{d:naive} and Lemma~\ref{l:naivefail}); this motivates the following definition.
\begin{definition}
	A \emph{pseudo{\textasteriskcentered}group} is a subset $\mathcal{S}\subset \Ph(X)$ satisfying Axioms~(\ref{i:identity})--(\ref{i:restrictions}) in Definition~\ref{d:pseudogroup}.
\end{definition}
This terminology was introduced by S.~Matsumoto in~\cite{Matsumoto}. For $S\subset\Ph(X)$, let $\langle S\rangle\subset\Ph(X)
$ denote the set consisting of finite compositions and inversions of elements in $S$, and let $S^*\subset\Ph(X)$ denote the set of partial homeomorphisms obtained from $S$ by composition, inversion, and restriction to open subsets; equivalently,
\[
S^*=\{\,s|_U : s\in\langle S\rangle,\  U\subset \dom s\ \text{open}\,\}
\]
and $S^*$ is the smallest pseudo{\textasteriskcentered}group containing $S$.

\begin{lemma}\label{l:generate}
	Let $\mathcal{G}\acts X$ be a pseudogroup and let $S\subset \mathcal{G}$. Then $S$ generates $\mathcal{G}$ if and only if, for every $g\in \mathcal{G}$ and $x\in\dom g$, there is an open neighbourhood $U$ of $x$ such that $g|_{U}\in S^*$.
\end{lemma}
\begin{proof}
	Let $\mathcal{H}\subset \Ph(X)$ denote the set of maps that result from combining families of maps in $S^*$ using Axiom~\ref{d:pseudogroup}(\ref{i:combination}). In other words, $\mathcal{H}$ is the pseudogroup generated by $S$. Then $S$ generates $\mathcal{G}$ if and only if $\mathcal{H}=\mathcal{G}$; that is, every map in $\mathcal{G}$ can be obtained as the combination of a family of maps in $S^*$, but this follows from the hypothesis and Axiom~\ref{d:pseudogroup}(\ref{i:combination})$'$. 
\end{proof}

\begin{corollary}\label{c:restrictionsstar}
	Let $S$ be a generating set for $\mathcal{G}\acts X$, let $d$ be a compatible metric on $X$,  let $g\in \mathcal{G}$, and let $K\subset \dom g$ be compact. Then there is $\epsilon>0$ such that, for every $x\in K$, the restriction of $g$ to $B_d(x,\epsilon)$ belongs to $S^*$.
\end{corollary}

\subsection{Equivalences}\label{ss:equivalences}
If we are to use pseudogroups to study foliated spaces, the right notion of isomorphism is that of \emph{equivalence}, sometimes also referred to as \emph{Haefliger} or \emph{\'etale equivalence}.

\begin{definition}\label{d:equiv}
	Let $\mathcal{G}\acts X$ and $\mathcal{H}\acts Y$ be pseudogroups. An \emph{equivalence} $\Phi\colon (X,\mathcal{G})\to (Y,\mathcal{H})$ is a collection of partial homeomorphisms $\Phi\subset\Ph(X,Y)$ satisfying the following conditions.
	\begin{enumerate}[(i)]
		\item \label{i:equivcover}$\{\,\dom \phi\mid\phi\in\Phi\,\}$ and $\{\,\im \phi\mid\phi\in\Phi\,\}$ are open coverings of $X$ and $Y$, respectively.
		\item \label{i:equivrest}If $\phi\in\Phi$ and $U$ is an open subset of $\dom \phi$, then $\phi|_{U}\in \Phi$.
		\item \label{i:equivcomb}Let $\phi\in\Ph(X,Y)$. If there is an open covering $\{U_i\}_{i\in I}$ of $\dom \phi$ such that $\phi|_{U_i}\in\Phi$ for every $i\in I$, then $\phi\in\Phi$.
		\item \label{i:equivcomp}If $g\in\mathcal{G}$, $h\in \mathcal{H}$, and $\phi\in\Phi$, then $h\phi g\in \Phi$.
		\item \label{i:equivgamma}If  $\phi,\psi\in\Phi$, then $\psi^{-1}\phi\in\mathcal{G}$ and $\psi \phi^{-1}\in \mathcal{H}$.
	\end{enumerate}
\end{definition}

The following properties follow immediately from the definition.
\begin{lemma}\label{l:equivproperties}
	Let $\Phi\colon (X,\mathcal{G})\to (Y,\mathcal{H})$ and $\Psi\colon (Y,\mathcal{H})\to (Z,\mathcal{I})$ be equivalences. Then the \emph{inverse} 
	\[
	\Phi^{-1}:=\{\, \phi^{-1}\mid\phi\in \Phi \,\}\subset \Ph(Y,X)
	\]
	and the \emph{composition}
	\[
	\Psi\circ\Phi:=\{\,\psi\circ\phi\mid\phi\in\Phi,\ \psi\in\Psi\,\}\subset\Ph(X,Z)
	\]
	are equivalences $(Y,\mathcal{H})\to (X,\mathcal{G})$ and $ (X,\mathcal{G})\to(Z,\mathcal{I})$, respectively.
\end{lemma}

\begin{lemma}\label{l:equivid}
	Let $\mathcal{G}\acts X$ be a pseudogroup and let $U\subset X$ be an open set that meets every $\mathcal{G}$-orbit. Then 
	\[
	\Phi:=\{\,g\in\mathcal{G}\mid \dom g\subset U\,\}
	\]
	is an equivalence $\Phi\colon (U,\mathcal{G|}_U)\to (X,\mathcal{G})$; in particular,  $\mathcal{G}$ is an equivalence $(X,\mathcal{G})\to (X,\mathcal{G})$.
\end{lemma}
\begin{proof}
	Item~(\ref{i:equivcover}) in Definition~\ref{d:equiv} follows from the assumption that $U$ meets every $\mathcal{G}$-orbit, whereas~(\ref{i:equivrest})--(\ref{i:equivgamma}) hold because $\mathcal{G}$ is a pseudogroup.
\end{proof}

Pseudogroup equivalences are maximal families in the following sense.

\begin{lemma}\label{l:equivmaximal}
	Let $\Phi, \Psi$ be equivalences $(X,\mathcal{G})\to (Y,\mathcal{H})$. If $\Phi \subset \Psi$, then $\Phi=\Psi$.
\end{lemma}
\begin{proof}
Let $\psi\in \Psi$. Since $\{\,\im \phi\mid\phi\in\Phi\,\}$ covers $Y$, there is a family $\{\phi_i\}_{i\in I} \subset \Phi$ such that $\{\im \phi_i\}_{i\in I}$ covers $\im \psi$. Moreover, $\phi_i\in \Psi$ by hypothesis, so $\phi_i^{-1}\psi\subset \mathcal{G}$ by~\ref{d:equiv}(\ref{i:equivgamma}), and then $\psi|_{\im \phi_i}= \phi_i(\phi_i^{-1}\psi)$ belongs to $\Phi$ by~\ref{d:equiv}(\ref{i:equivcomp}). Finally, $\psi$ is the combination of the family $\{\psi|_{\im \phi_i}\}_{i\in I}$, so $\psi\in \Phi$ by~\ref{d:equiv}(\ref{i:equivcomb}). This shows $\Psi\subset \Phi$ because $\psi$ was chosen arbitrarily.
\end{proof}

\begin{lemma}\label{l:morphismgenerate}
	Let $\mathcal{G}\acts X$ and $\mathcal{H}\acts Y$ be pseudogroups, and let $\Sigma\subset\Ph(X,Y)$ be a family of maps such that
	\begin{enumerate}[(i)]
		\item \label{i:equivgenmeetg}$\bigcup_{\phi\in\Sigma} \dom \phi\subset X$ meets every $\mathcal{G}$-orbit;
		\item \label{i:equivgenmeeth}$\bigcup_{\phi\in\Sigma} \im \phi\subset Y$ meets every $\mathcal{H}$-orbit; and,
		\item \label{i:equivgengamma}if $\phi$, $\psi\in\Sigma$, $g\in \mathcal{G}$, and $h\in\mathcal{H}$, then $\psi^{-1}h\phi\in\mathcal{G}$ and $\psi g\phi^{-1}\in \mathcal{H}$.
	\end{enumerate}
	Then there is a unique equivalence $\Phi\colon (X,\mathcal{G})\to (Y,\mathcal{H})$ containing $\Sigma$.
\end{lemma}
\begin{proof}
Let $\Phi\subset \Ph(X,Y)$ consist of the combinations of maps of the form $h\sigma g$, where $g\in \mathcal{G}$, $h\in \mathcal{H}$, and $\sigma \in \Sigma$; then $\Phi$ is an equivalence: Axiom~\ref{d:equiv}(\ref{i:equivcover}) follows from~(\ref{i:equivgenmeetg}) and~(\ref{i:equivgenmeeth}), \ref{d:equiv}(\ref{i:equivrest})--(\ref{i:equivcomp})  follow from the definition of $\Phi$, and ~\ref{d:equiv}(\ref{i:equivgamma}) follows from~(\ref{i:equivgengamma}).  Finally, Lemma~\ref{l:equivmaximal} yields uniqueness.
\end{proof}

We will refer to the equivalence given by Lemma~\ref{l:morphismgenerate} as the equivalence \emph{generated} by $\Sigma$. We say that two pseudogroups are \emph{equivalent} if there is an equivalence from one to the other; this is a reflexive, symmetric, and transitive relation by Lemmas~\ref{l:equivproperties} and~\ref{l:equivid}. The reader should be mindful that equivalence of pseudogroups is a very lax condition, as the next example shows.
\begin{example}[{\cite[{p.\ 277}]{Haefliger}}]\label{e:rz}
	Let $\mathcal{G}$ be the pseudogroup on $\mathbb{R}$ generated by the translation $t\mapsto t+1$, and let $\mathcal{H}$ be the pseudogroup on $\mathbb{S}^1$ generated by the identity map. Consider the natural projection map $\pi\colon\mathbb{R}\to \mathbb{R}/\mathbb{Z}\cong\mathbb{S}^1$. Then
	\[
	\Phi:=\{\,\pi|_{U}\mid U\subset \mathbb{R}\ \text{open},\ \pi|_{U}\colon U\to \phi(U)\ \text{is a homeomorphism}\,\}
	\]
	is an equivalence from $(\mathbb R,\mathcal{G})$ to $(\mathbb{S}^1,\mathcal{H})$.
\end{example}

Finally, if $X$ and $Y$ are $C^i$-manifolds for some $i\in \mathbb{N}\cup\{\infty,\omega\}$\footnote{The notation $C^\omega$ means that the manifold or map is analytic}, we say that a family $A\subset \Ph(X,Y)$ is $C^i$ if all the maps in $A$ are $C^i$ in the usual sense. In this way we obtain a definition of $C^i$-pseudogroups and equivalences, and  the notion of being a $C^i$-pseudogroup is then invariant by $C^i$-equivalences. Similarly, if $X$ and $Y$ are affine manifolds, we can define affine pseudogroups and equivalences, and being an affine pseudogroup is invariant by affine equivalences.

\subsection{Compact generation}\label{ss:compactlygenerated}

\begin{definition}
	Let $\mathcal{G}\acts X$ be a pseudogroup.
	A \emph{system of compact generation}   is a triple $(U,F,\widetilde F)$, where 
	\begin{enumerate}[(i)]
		\item $U$ is a relatively compact open set of $X$ meeting every $\mathcal{G}$-orbit,
		\item both $F\subset\mathcal{G}|_U$ and $\widetilde F\subset\mathcal{G}$ are finite and symmetric,
		\item $F$ generates $\mathcal{G}|_U$, and
		\item there is a bijection $f\mapsto \tilde f$ ($f\in F$, $\tilde f\in\widetilde F$) 	where $\tilde f$ is an extension of $f$ and $\overline{\dom f}\subset \dom(\tilde f)$ for every $f\in F$. 
	\end{enumerate}

\end{definition}

We say that $\mathcal{G}$ is \emph{compactly generated} if it admits a system of compact generation; note that compact generation implies that $X$ is locally compact. This property is invariant by equivalences~\cite{Haefliger2}.
The main family of examples of compactly generated pseudogroups, which moreover gave birth to the definition, consists of holonomy pseudogroups of compact foliated spaces (see next section). As a simpler example, we could mention the pseudogroup generated by the action of a finitely generated group on a compact space.

Contrary to most of the references in the subject, we consider the set of extensions $\widetilde{F}$ as part of the generating set. 
Note that the symmetry condition of both $F$ and $\widetilde{F}$ is included for simplicity.

From now on, for every map $f\in \langle F\rangle$, $f=f_{n} \cdots f_{1}$ with $f_{i}\in F$, we denote by $\tilde{f}\in\langle\widetilde{F}\rangle$ the composition $\tilde f_{n}\cdots \tilde f_{1}$. For notational convenience, we will assume from now on that $\tilde{f}^{-1}=\widetilde{f^{-1}}$. Properly speaking, the map $\tilde f$ depends not only on $f$, but on the representation $f=f_{n} \cdots f_{1}$; we will incur in this slight abuse of notation anyway because this subtlety will be of no relevance to our proofs. 

\begin{lemma}\label{l:systemcg}
	Let $\mathcal{G}$ be a compactly generated pseudogroup, let $x\in X$, and let $S$ be a generating set. Then there is a system of compact generation $(U,F,\widetilde F)$ with $x\in U$ and $\widetilde F^*\subset S^*$.
\end{lemma}

\begin{proof}
	We begin by showing that we can choose $(U,F,\widetilde F)$ with $x\in U$. Indeed, let $(V,H,\widetilde H)$ any system of compact generation and let $g\in\mathcal{G}$ be any map with $x\in\dom g$ and $g(x)\in V$. Let $W,W'$ be relatively compact open neighborhoods of $x$ with $\overline{W}\subset W'\subset \dom g$. Then $(V\cup W, H\cup \{g|_{W}\}, \widetilde{H}\cup\{g|_{W'}\})$ is a system of compact generation.
	
	Let us show now that we may take $(U,F,\widetilde F)$ with $\widetilde{F}^*\subset S^*$. Let $(U,H,\widetilde H)$ satisfy $x\in U$. Write $H=\{f_i\}_{i\in I}$; then, for every $i\in I$, there is a finite open cover $\{V_{i,j}\}_{j\in J_i}$ of $\overline{\dom f_i}$ and a shrinking $\{W_{i,j}\}_{j\in J_i}$ such that $\tilde f|_{V_{i,j}}\in S^*$ for every $j\in J_i$ by Corollary~\ref{c:restrictionsstar}. Then 
	\[
	(U,F,\widetilde F):=(U,\{\tilde f_i|_{W_{i,j}\cap U}\},\{\tilde f_i|_{V_{i,j}}\})
	\]
	is a system of compact generation satisfying the desired conditions.
\end{proof}

\subsection{Foliated spaces}\label{ss:foliated}

Let $X$ be a  Polish space and let $\mathcal{F}$ be a partition of $X$. Then $(X,\mathcal{F})$ is a \emph{foliated space} of \emph{leafwise class} $C^k$ ($k\in\mathbb{N}\cup\{\infty\}$)  and dimension $n\in \mathbb{N}$ if $X$  admits an atlas of charts $(U_i,\phi_i)$, where $\{U_i\}$ is an open covering of $X$ and the maps $\phi_i$ are homeomorphisms $\phi_i\colon U_i\to \mathbb{R}^n\times Z_i$ (for $Z_i$ Polish), and with coordinate changes  of the form
\[
\phi_i\phi_j^{-1}(x_j,z_j)=(x_i(x_j,z_j),z_i(z_j)),
\] 
where $z_i\colon \phi_j(U_i\cap U_j)\to Z_i$ is continuous and $x_i\colon \phi_j(U_i\cap U_j)\to\mathbb{R}^n$ is of class $C^k$ on every plaque. Remember that the \emph{plaques} of the chart $(U_i,\phi_i)$ are the sets $\phi_i^{-1}(\mathbb{R}^n\times\{z_i\})$. Moreover, we require that the equivalence relation induced by $\mathcal{F}$ coincides with the transitive closure of the relation ``being in the same plaque''; it follows that $\mathcal{F}$ partitions $X$ into subsets which, when endowed with an appropriate topology called the \emph{leaf topology}, become connected  $C^k$-manifolds of dimension $n$: the \emph{leaves} of the foliated space.  If $(X,\mathcal{F})$ is a foliation of dimension $n$, codimension $m$, and class $C^{k,l}$ (see~\cite[p.~32]{CandelConlon2000-I}), then it is an $n$-dimensional foliated space of leafwise class $C^k$, the transversal models $Z_i$ are $C^l$-manifolds of dimension $m$,  and the maps $z_i$ are of class $C^l$.

The \emph{holonomy pseudogroup} serves as a dynamical model for the foliated space $X$: let  $\{(U_i,\phi_i)\mid i\in I\}$ be a locally finite atlas, let $p_i$ denote the composition of $\phi_i$ with the projection $\mathbb{R}^n\times Z_i\to Z_i$, and let $Z=\coprod_{i\in I} Z_i$. The transversal components of the change of coordinate maps
\[
h_{i,j}\colon p_j(U_i\cap U_j)\to p_i(U_i\cap U_j)  ,\qquad h_{i,j}(z_j)=z_i(z_j)
\]
generate a pseudogroup in $Z$, called the \emph{holonomy pseudogroup} of $X$. Note that we are only considering holonomy pseudogroups induced by locally finite atlases so that the transversal space $Z$ is Polish and the pseudogroup is countably generated.

The holonomy pseudogroup depends on the choice of atlas  $\{(U_i,\phi_i)\mid i\in I\}$, but different choices give rise to equivalent pseudogroups (in the case of foliations of class $C^{k,l}$, $C^l$-equivalent pseudogroups). Thus, from now on, we restrict ourselves to considering properties of pseudogroups that are invariant by equivalences; this justifies our abuse of language when we talk about ``the'' holonomy pseudogroup of $X$.

\begin{lemma}[{\cite{Haefliger}}]
	If $X$ is a compact foliated space, then its holonomy pseudogroup is compactly generated. 
\end{lemma}

If $X$ is a foliation of codimension $m$ and it admits an atlas such that the transversals have an affine structure and the maps $h_{i,j}$ are all affine, then it is a \emph{transversally affine foliation} and its holonomy pseudogroup is also affine.

A \emph{matchbox manifold} is a compact foliated space admitting an atlas with totally disconnected transversals.

\subsection{Toral linked twist maps}\label{s:linked}

Let $\mathbb{T}^2:=\mathbb{R}^2/\mathbb{Z}^2$ be the $2$-torus, whose points we will simply denote as pairs $(x,y)$, where $x, y\in\mathbb{R}^2$. For an interval $A=[a,b]$ with $0\leq a< b\leq 1$, let $H_A$ be the horizontal closed annulus defined by
\[
H_A=\{(x,y)\in \mathbb{T}^2\mid a\leq  y\leq b\},
\]
and let $V_A$ be the corresponding vertical closed annulus
\[
V_A=\{(x,y)\in \mathbb{T}^2\mid a\leq  x\leq b\}.
\]

For any integer $m>1$, we have the horizontal and vertical twist maps, defined  on $H_A$ and $V_A$, respectively, by
\[
(x,y)\mapsto (x+\phi_m(y),y),\qquad (x,y)\mapsto (x,y+\phi_m(x)),
\]
where 
\[
\phi_m(t)=\frac{m(t-a)}{(b-a)}
\]
is the only affine map satisfying 
\[
\phi_m(a)=0,\qquad \phi_m(b)=m.
\] 
Note that $\phi_m$ depends of course on the choice of interval $A$, but we leave it implicit to avoid cumbersome notation. Toral linked twists can be constructed with more general maps $\phi_m$ (see e.g.~\cite{Devaney}), but in this paper we  restrict our attention to the linear case  for the sake of simplicity.

A \emph{toral linked twist} is the composition of horizontal twist maps on a finite number of horizontal annuli with  vertical twists on a finite number of vertical annuli. Let $\widehat H_1,\ldots,\widehat H_k$ be a collection of closed intervals in $[0,1]$ such that every intersection $\widehat H_i\cap \widehat H_j$ with $i\neq j$ consists of at most one common endpoint, and let $\widehat V_1,\ldots,\widehat V_l$ be another such collection.  Let $H_1,\ldots, H_k$ be the horizontal closed annuli induced by the intervals $\widehat H_1,\ldots,\widehat H_k$; similarly, let $V_1,\ldots,V_l$ be the vertical closed annuli induced by $\widehat V_1,\ldots,\widehat V_l$. Choose two sequences of positive integers 
\[m_1,\ldots, m_k,\qquad n_1,\ldots,n_l,\]
 and, for $i=1,\ldots,k$, let $h_i$ denote the horizontal $m_i$-twist map on $H_i$,
\[
h_i(x,y)=(x+\phi_{m_i}(y),y).
\]
 Define the vertical $n_i$-twist maps  $v_j$, $1\leq j\leq l$, similarly; see Figure~\ref{fig:twist} for an illustration.

We combine the horizontal twists $h_i$ into one map $T_h$ and the vertical twists into another map $T_v$ as follows:
\begin{align*}
	T_h(x,y)&=\begin{cases}
		h_i(x,y) &\qquad\text{if}\ (x,y)\in H_i\ \text{for some}\ 1\leq i\leq k,\\
		(x,y) &\qquad\text{else}.
	\end{cases}\\
	T_v(x,y)&=\begin{cases}
		v_i(x,y) &\qquad\text{if}\ (x,y)\in V_i\ \text{for some}\ 1\leq i\leq l,\\
		(x,y) &\qquad\text{else}.
	\end{cases}
\end{align*}
The linked twist map corresponding to our choice of intervals and integers is then $T=T_v\circ T_h$. 

We review some of the basic properties that will be of use later. First, note that $T$ is the identity on $\mathbb{T}^2\setminus M$, where
\[
M=H_1\cup\cdots\cup H_k \cup V_1\cup \cdots\cup V_l.
\]
Also, $T$ is affine (hence smooth) on $\mathbb{T}^2\setminus\Delta$, where
\[
\Delta = \partial H_1\cup\cdots\cup\partial H_k \cup T_h^{-1}(\partial V_1)\cup \cdots\cup T_h^{-1}(\partial V_l)
\]
and $\partial$ denotes the topological boundary.
Finally, we will also employ the following result.
\begin{theorem}[{\cite[Thm.\ A]{Devaney80}}]\label{t:twisttt}
	The restriction of the toral linked twist map $T$ to $M$ is topologically transitive  and sensitive to initial conditions.
\end{theorem}

\begin{figure}[tbh]
	\includegraphics[width=\textwidth]{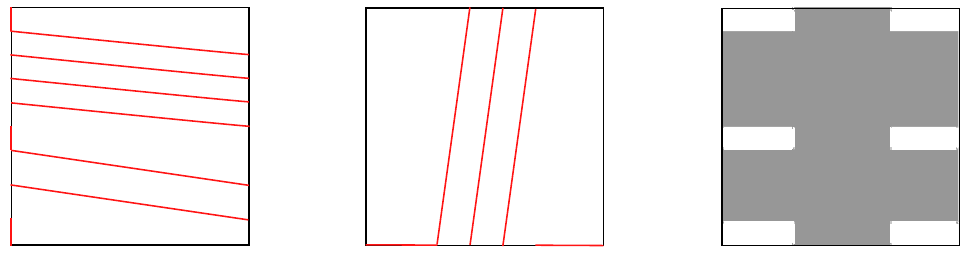}
	\caption{The horizontal  (left, $T_h$)\ and vertical (middle, $T_v$) components  of a linked twist; the red lines are the images of the circles represented by the vertical (left) and horizontal (middle) boundary segments. They involve two and one intervals, respectively. The shaded area on the right represents $M$.}
    \label{fig:twist}
\end{figure}

\subsection{Equicontinuous pseudogroups}

\begin{definition}[{\cite[{Def.~7.1}]{AlvarezCandel}}]\label{d:qlm}
	Let $X$ be a topological space, and let $\{(U_i,d_i)\mid i\in I\}$ be a family of metric spaces such that $\{U_i\}$ is an open covering of $X$ and every $d_i$ is a compatible metric on $U_i$.
	We say that $\{(U_i,d_i)\}$ is a \emph{cover of $X$ by quasi-locally equal metric spaces} if there is an assignment $\epsilon\mapsto \delta(\epsilon)$ such that, for every $i,j\in I$, every point $z\in U_i\cap U_j$ has an open neighborhood $U_{i,j,z}\subset U_i\cap U_j$ satisfying
	\[
	d_i(x,y)<\delta(\epsilon)\quad\Longrightarrow\quad d_j(x,y)<\epsilon
	\]
	for every $\epsilon>0$ and $x,y\in U_{i,j,z}$.
	Two such covers $\{(U_i,d_i)\}$ and $\{(V_j,d'_j)\}$ are \emph{equivalent} if their union is again a cover by quasi-locally equal metric spaces; an equivalence class of covers is called a \emph{quasi-local metric space}.
\end{definition}

\begin{proposition}[{\cite[{Thm.~15.1}]{AlvarezCandel}}] 
	If $X$ is Hausdorff and paracompact, then every cover by quasi-locally equal metric spaces is equivalent to a metric; that is, equivalent to a cover of the form $\{(X,d)\}$.
\end{proposition}

\begin{definition}[{\cite[{Def.~8.4}]{AlvarezCandel}}]\label{d:eqpsg}
	Let $\mathcal{G}$ be a pseudogroup acting on a Polish space $X$. We say that $\mathcal{G}$ is \emph{equicontinuous} if there is a cover by quasi-locally equal metric spaces $\{(U_i,d_i)\mid i\in I\}$, a generating pseudo{\textasteriskcentered}group $\mathcal{S}$, and an assignment $\epsilon\mapsto \delta(\epsilon)$ such that
	\[
	d_i(x,y)<\delta(\epsilon)\quad\Longrightarrow\quad d_j(sx,sy)<\epsilon
	\]
	for every $i,j\in I$, $s\in \mathcal{S}$, $x,y\in\dom s\cap U_i$ and $sx,sy\in U_j$.
\end{definition}

Note that if the above condition is fulfilled with a cover by quasi-locally equal metric spaces, then it is also fulfilled with any other equivalent cover, so we can regard equicontinuity as a property of the quasi-local metric.

 By the previous results, Definition~\ref{d:eqpsg} is equivalent to the following:

\begin{definition}
	A pseudogroup $\mathcal{G}\acts X$ is \emph{equicontinuous} if there is a generating pseudo{\textasteriskcentered}group $S$, a compatible metric $d$, and an assignment $\epsilon\mapsto\delta(\epsilon)$ such that
	\[
	d(x,y)<\delta(\epsilon)\qquad \Longrightarrow \qquad d(sx,sy)<\epsilon
	\]
	for every $s\in S$ and $x,y\in \dom s$.
\end{definition}

\begin{proposition}[{\cite[{Lem.~8.8}]{AlvarezCandel}}] \label{p:equicontinuityinvariant}
	Being equicontinuous is invariant by equivalences of pseudogroups.
\end{proposition}

\section{Pseudogroup dynamics}
\subsection{Sensitivity and chaos}\label{ss:sensitivityandchaos}
The aim of this section is to introduce our definition of  Devaney chaos for pseudogroups; first, we need to obtain  suitable analogues of conditions~(\ref{i:devtt})--(\ref{i:devsens}) in Definition~\ref{d:dev}.
\begin{definition}
	A pseudogroup $\mathcal{G}\curvearrowright X$ is \emph{topologically transitive} if, for all non-empty open subsets $U,V\subset X$, there is some $g\in \mathcal{G}$ with
	\[
	g(U)\cap V\neq \emptyset.
	\]
\end{definition}

\begin{definition}
	A pseudogroup $\mathcal{G}\curvearrowright X$ is \emph{point transitive}  if there is some $x\in X$ such that $\mathcal{G}x$ is dense in $X$.
\end{definition}

\begin{lemma}[{cf.~\cite[Prop.~1.1]{Silverman} and~\cite[Prop.~4.5]{Flores}}]\label{l:pt}
	Point transitivity implies topological transitivity for every pseudogroup $\mathcal{G}\acts X$; if $X$ is separable and Baire, then the converse also holds.
\end{lemma}

\begin{proof}
	To show that point transitivity implies topological transitivity, let $\mathcal{G}x$ be a dense orbit and let $U,V$ be non-empty open sets. Then there are $g,h\in\mathcal{G}$ with $g(x)\in U$, $h(x)\in V$, and therefore $hg^{-1}(U)\cap V\neq\emptyset$.
	
	Suppose now that $\mathcal{G}$ is topologically transitive but there is no dense orbit, and let $\{U_n\}_{n\in\mathbb{N}}$ be a countable base for $X$. For every $x\in X$, there is some $U_{n(x)}$ such that $\mathcal{G}x\cap U_{n(x)}=\emptyset$. For each $n$, $\mathcal{V}_{n}=\bigcup_{g\in\mathcal{G}} g(U_{n})$	is a dense open set because $\mathcal{G}$ is topologically transitive, so $X\setminus \mathcal{V}_{n(x)}$ is a closed and nowhere dense set containing $x$. Thus, $X=\bigcup_{n\in\mathbb{N}}X\setminus \mathcal{V}_{n}$ is a countable union of closed and nowhere dense sets, contradicting the assumption that $X$ was Baire.
\end{proof}

Recall that we are working with Polish (hence, Baire) spaces, so topological transitivity and point transitivity coincide.

Regarding density of periodic points, one may be tempted to use the following naive definition: $\mathcal{G}\acts X$ has density of periodic points if the union of finite $\mathcal{G}$-orbits is dense in $X$. Unfortunately, Example~\ref{e:rz} shows that this condition is not invariant by equivalences, so we need to reformulate it as follows:

\begin{definition}\label{d:dpo}
	A pseudogroup $\mathcal{G}\acts X$ has \emph{density of periodic orbits} if there is an open set $U\subset X$ meeting every $\mathcal{G}$-orbit and such that the set of finite $\mathcal{G}|_U$-orbits is dense in $U$.
\end{definition}

\begin{lemma}\label{l:dppinvariant}
	Having density of periodic orbits is invariant by equivalences.
\end{lemma}
\begin{proof}
	Let $\mathcal{G}\acts X$ be a pseudogroup with density of periodic orbits, let $U\subset X$ be an open set meeting every $\mathcal{G}$-orbit and such that the finite $\mathcal{G}|_U$-orbits are dense, and let $\Phi\colon (X,\mathcal{G}) \to (Y,\mathcal{H})$ be an equivalence of pseudogroups.
	
	Since $U$ is a paracompact space, we can find a subset $\Phi_0\subset \Phi$ such that $\{\dom \phi\mid \phi\in\Phi_0\}$ is a locally finite family and $\bigcup_{\phi\in\Phi_0} \dom \phi=U$. We claim that 
	\[
	V=\bigcup_{\phi\in\Phi_0} \im \phi
	\]
	satisfies the statement in Definition~\ref{d:dpo}. Clearly, $V$ meets every $\mathcal{H}$-orbit, so let us prove that finite $\mathcal{H}|_V$-orbits are dense in $V$.
	
	Let $W\subset V$ be any open set and choose $\phi_0\in\Phi_0$ with $\im \phi_0\cap W\neq\emptyset$.  $\phi_0^{-1}(W)\subset U$,  so there is some $y\in \phi^{-1}_0(W)$ such that $\mathcal{G}|_U(y)$ is finite.
	Since $\{\dom \phi\mid \phi\in\Phi_0\}$ is a locally finite family and $\mathcal{G}|_U(y)$ is finite, there are only finitely many maps in $\Phi_0$ defined on $\mathcal{G}|_U(y)$, so 
	\[
	A_y:=\{\phi(z)\mid \phi\in\Phi_0, z\in \mathcal{G}|_U(y)\}
	\]
	is a finite set. 
	
	Let us show that $\mathcal{H}|_V(A_y)= A_y$. For every $h\phi(z)$, $h\in \mathcal{H}|_{V}$, there is some $\psi\in \Phi_0$ with $h\phi(z)\in \im \psi$, and therefore 
	\[
	z':=\psi^{-1} h\phi(z)\in \mathcal{G}|_U(z)=\mathcal{G}|_U(y)
	\] 
	by~\ref{d:equiv}(\ref{i:equivgamma}). Hence $h\phi(z)= \psi(z')$ with $z'\in \mathcal{G}|_U(y)$, showing that $\mathcal{H}|_V(A_y)= A_y$. 
	
	We have proved that every open set $W\subset V$ meets an $\mathcal{H}|_V$-invariant finite set $A_y$, so $\mathcal{H}$ has density of periodic orbits.
\end{proof}

\begin{corollary}\label{c:dpocg}
	If $\mathcal{G}\acts X$ has density of periodic orbits and $W\subset X$ is a relatively compact open set meeting every orbit, then the finite  $\mathcal{G}|_W$-orbits are dense in $W$.
\end{corollary}
\begin{proof}
Let $U\subset X$ be an open set satisfying the statement of Definition~\ref{d:dpo}, and let $F\subset \mathcal{G}$ be a finite set satisfying $\overline{W}\subset \{\im f\mid f\in F\}$ and $\bigcup_{f\in F}\dom f\subset U$. Since $W$ meets every $\mathcal{G}$-orbit, $V:=\bigcup_{f\in F}f^{-1}(W)$ also meets every $\mathcal{G}$-orbit; moreover, $V\subset U$, so the set of finite $\mathcal{G}|_V$-orbits is dense in $V$. Hence $\Phi_0=\{f|_{f^{-1}(W)} : f\in F\}$ is a finite set  satisfying $W=\bigcup_{\phi\in\Phi_0}\im \phi$, and, arguing as in the proof of the previous proposition, we get that the set of finite $\mathcal{G}|_W$-orbits is dense in $W$.
\end{proof}

Finally, we come to the definition of sensitivity for pseudogroups. A naive approach would suggest the following definition: 
\begin{definition}[Naive sensitivity]\label{d:naive}
	$\mathcal{G}\acts X$ is sensitive if there is $c>0$ such that, for every $x\in X$ and $r>0$, there are $g\in \mathcal{G}$ and $y\in B(x,r)$ with 
	$x, y\in\dom g$ and
	 $d(g(x),g(y))\geq c$.  
\end{definition}

However, the following lemma shows that this condition is too weak to model our idea of sensitivity:

\begin{lemma}\label{l:naivefail}
	Any topologically transitive pseudogroup $\mathcal{G}$ on a perfect  Polish space $X$ satisfies Definition~\ref{d:naive}.
\end{lemma}
\begin{proof}
Since $X$ is perfect, we can choose two non-empty open sets $W_1,W_2$ and $c>0$ satisfying 
\begin{equation*}
d(W_1,W_2)\geq 2c.
\end{equation*} 
Let $x\in X$ and $r>0$ be arbitrary, and choose  $0<s<r$ such that  $V=B(x,r)\setminus \overline{B(x,s)}$ is non-empty; such  $s$ exists because $X$ is perfect. By topological transitivity, there are maps $g_1$ and $g_2$  in $\mathcal{G}$ such that 
\[
g_1(V)\cap W_1\neq \emptyset,\qquad g_2(V)\cap W_2\neq \emptyset;
\]
by restricting to open subsets if necessary, we may assume $\dom g_1,\dom g_2\subset V$ and $\im g_i\subset W_i$ for $i=1,2$,. Let $h_i$, $i=1,2$, be the partial map with domain 
\[
\dom h_i = \dom g_i\cup B(x,s)
\]
defined by
\[
h_i|_{\dom g_i}=g_i,\qquad h_i|_{B(x,s)}=\id_{B(x,s)}.
\]
Axiom~\ref{d:pseudogroup}(\ref{i:combination}) ensures that $h_i$ belongs to $\mathcal{G}$.
By construction, we have that $x,y_1\in\dom h_1$, $x,y_2\in\dom h_2$ and $d(h_1(y_1),h_2(y_2))\geq 2c.$
The triangle inequality now yields that, for some $i\in \{1,2\}$, we have
\[
d(h_i(x),h_i(y_i))= d(x,h_i(y_i))\geq c.\qedhere
\]
\end{proof}

This argument shows that the problem originates from Axiom~\ref{d:pseudogroup}(\ref{i:combination}) (closure under combinations). Following the ideas in~\cite{AlvarezCandel}, which in turn can be traced back to previous works (see~\cite{HectorHirsch,Matsumoto}), we phrase our definition of sensitivity for pseudogroups in terms of generating pseudo{\textasteriskcentered}groups. We also quantify over all compatible metrics in order to make it a topological condition.

\begin{definition}\label{d:sensitivity}
Given a metric $d$ on $X$ and a generating pseudo{\textasteriskcentered}group $\mathcal{S}$ for a pseudogroup $\mathcal{G}\curvearrowright X$,
we say that $\mathcal{G}\curvearrowright X$ is ($\mathcal{S}$,$d$)-\emph{sensitive to initial conditions} if there is a \emph{sensitivity constant} $c:=c(\mathcal{S},d)>0$ such that, for every $x\in X$ and $r>0$, there are  $s\in \mathcal{S}$ and $y\in \dom s$ with
\[
d(x,y)<r\quad \text{and} \quad d(sx,sy)\geq c.
\]
We say that $\mathcal{G}\curvearrowright X$ is \emph{sensitive to initial conditions} if it is  ($\mathcal{S}$,$d$)-sensitive to initial conditions for every choice of  $\mathcal{S}$ and $d$.
\end{definition}

Note that $c$ clearly varies with the choice of metric and pseudo{\textasteriskcentered}group: if we fix $d$ and choose a sequence of pseudo{\textasteriskcentered}groups $S_n$ such that  all sets $\im s_n$ ($s\in S_n$) have diameter less than $1/n$, then $c(d,S_n)\downarrow 0$.

In order to provide some intuition for this definition, let's consider the case of a $\mathbb{Z}$-action induced by a homeomorphism $f$ on some topological space $X$. If the $\mathbb{Z}$-action is sensitive to initial conditions (w.r.t.~the classical definition for group actions), there is $c>0$ so that, for every $x\in X$ and $\delta>0$, there are $y\in B(x,r)$ and $z\in \mathbb{Z}$ satisfying $d(f^z(x),f^z(y))>c$. However, the absolute value of $z$ will diverge as $\delta\to 0$. In the case of the pseudogroup induced by this action, it might happen that, for some generating pseudo{\textasteriskcentered}group $\mathcal{S}$, the domain of every map $s\in \mathcal{S}$ that agrees  with $f^z$ on some open set is so small that it does not contain any such $y$, and therefore the maps in $\mathcal{S}$ cannot bear witness to the sensitivity of the action. This is precisely why we must quantify over all generating pseudo{\textasteriskcentered}groups. Coming back to the example at hand, given a covering $\{U_i\}$ of $X$ by open sets, the restrictions $\{f|_{U_i}\}$ are a generating pseudo{\textasteriskcentered}group for the pseudogroup induced by the action. Note that, by composing maps, the diameters of the domains of the compositions might grow smaller and smaller. In this case, sensitivity means that, for every covering $\{U_i\}$, there is some positive $c$ (which depends on the cover) so that, for every point $x\in X$, we can find points $y$ arbitrarily close to $x$ that have compositions $f|_{U_{i_n}}\circ \cdots\circ f|_{U_{i_1}}$  defined on $x$ and $y$ and satisfying
\[
d(f|_{U_{i_n}}\circ \cdots\circ f|_{U_{i_1}}(x), f|_{U_{i_n}}\circ \cdots\circ f|_{U_{i_1}}(y))>c.
\]

Having obtained analogues of \ref{d:dev}(\ref{i:devtt})--(\ref{i:devsens}), we can introduce our definition of Devaney chaos for pseudogroups.

\begin{definition}
	A pseudogroup $\mathcal{G}\acts X$ is \emph{chaotic} if it is topologically transitive, has density of periodic orbits, and is sensitive to initial conditions.
\end{definition}

\subsection{Equicontinuous points}
In this subsection we will generalize to the setting of pseudogroups some dynamical notions expressing regularity; we will need them in order to prove the Auslander--Yorke dichotomy.
\begin{definition}
	A point $x\in X$ is \emph{$(S,d)$-equicontinuous} for a generating set $S$ and a metric $d$ on $X$ if there is an assignment $\epsilon\mapsto\delta(\epsilon)$ so that
	\[
	d(x,y)<\delta(\epsilon)\quad\Longrightarrow\quad d(s(x),s(y))<\epsilon
	\]
	for every $s\in S^*$ and $y\in X$ with $x,y\in \dom s$.
	We say that a point $x$ is \emph{equicontinuous} if it is equicontinuous for some choice of $S$ and $d$.
\end{definition}

We will refer to any assignment $\epsilon\mapsto\delta$ satisfying the above condition as a \emph{modulus of equicontinuity} for $(S,d)$.

\begin{definition}
	The pseudogroup $\mathcal{G}\acts X$ is \emph{almost equicontinuous} if there are $S$ and $d$ so that the set of $(S,d)$-equicontinuous points is dense in $X$.
\end{definition}

\begin{lemma}
	If $x\in X$ is $(S,d)$-equicontinuous, then every $y\in \mathcal{G}x$ is $(S,d)$-equicontinuous (perhaps with a different modulus).
\end{lemma}

\begin{proof}
	Let $y\in \mathcal{G}x$. Since $S$ is a generating set for $\mathcal{G}$, there is some $s\in S^*$ so that $s(x)=y$. For every $\epsilon>0$, let $\delta'(\epsilon)>0$ be small enough so that
	\[
	B(y,\delta'(\epsilon))\subset \im s,\qquad s^{-1}(B(y,\delta'(\epsilon)))\subset B(x,\delta(\epsilon)).
	\] Then, for every $t\in S^*$, the restrictions of $t$ and $t s s^{-1}$ to $\dom t \cap B(y,\delta'(\epsilon))$ coincide, and thus
	\[
	t(B(y,\delta'(\epsilon)))=tss^{-1}(B(y,\delta'(\epsilon)))\subset ts(B(x,\delta(\epsilon))).
	\] 
	Since $ts\in S^*$ and $\delta\mapsto\epsilon$ is a modulus of equicontinuity for $(S,d)$, we obtain
	\[
	t(B(y,\delta'(\epsilon)))\subset B(t(y),\epsilon). \qedhere
	\]
\end{proof}

\subsection{Dynamics and compact generation}
We begin with some preliminary results for compactly generated pseudogroups.  The following proposition reveals that, for every system of compact generation $(U,F,\widetilde F)$ and every point $x\in U$, either the pseudo{\textasteriskcentered}group $\langle \widetilde F\rangle$ displays sensitivity to initial conditions on $x$, or every map in $\langle F\rangle$ defined on $x$ has an extension in $\langle \widetilde F\rangle$ whose domain contains a ball of a fixed radius $\rho>0$.

\begin{proposition}\label{p:halo}
	Let $\mathcal{G}$ be a compactly generated pseudogroup on $X$, let $d$ be a compatible metric, let $(U,F,\widetilde F)$ be a system of compact generation, and let
	\[
	\sigma:=\sigma(U,F,\widetilde F)=\sup\{r>0\mid B(u,r)\subset \dom \tilde f\ \ \forall f\in F,\ u\in\dom f\}>0.
	\]
	Then, for every $x\in U$, either
	\begin{enumerate}[(i)]
		\item \label{i:halo} there is $\rho>0$ such that $B(x,\rho)\subset \dom \tilde f$ for every $f\in \langle F\rangle$ with $\dom f\cap B(x,\rho)\neq \emptyset$; or,
		\item \label{i:nothalo} for every $r>0$, there are $y\in B(x,r)$ and $\tilde f\in\langle \widetilde F\rangle$ satisfying
		\[
		x,y\in\dom \tilde f,\qquad d(\tilde f(x),\tilde f(y))\geq \sigma/2.
		\]
	\end{enumerate}
\end{proposition}
\begin{proof}
	Suppose that~(\ref{i:halo}) does not hold, so that, for every $r>0$, there are $f\in \langle F\rangle$ and $y,z\in B(x,r)$ satisfying $z\in\dom f\subset \dom \tilde f$, $y\notin \dom \tilde f$.
	Let $f=f_n\cdots f_1$, where $f_i\in F$ for $i=1,\ldots,n$. Let $j$ be the largest index $0\leq j<n$ satisfying $B(x,r)\subset \dom \tilde f_j\cdots\tilde f_1$. This means that there is $y\in \dom\tilde f_j\cdots\tilde f_1$ such that
	\[
	\tilde f_j\cdots\tilde f_1(y)\notin \dom \tilde f_{j+1}.
	\]
	But $z\in \dom \tilde f_{j+1}\cdots\tilde f_1$, so
	\[
	d(\tilde f_j\cdots\tilde f_1(y),\tilde f_j\cdots\tilde f_1(z))\geq \sigma
	\]
	by the definition of $\sigma$.  Now the triangle inequality yields  either
	\[d(\tilde f_j\cdots\tilde f_1(x),\tilde f_j\cdots\tilde f_1(y))\geq \sigma/2\quad \text{or} \quad d(\tilde f_j\cdots\tilde f_1(x),\tilde f_j\cdots\tilde f_1(z))\geq \sigma/2.\qedhere\]
\end{proof}

The following result, which will be of use later, is a generalization of the well-known fact that equicontinuity and uniform equicontinuity agree for actions on compact spaces.

\begin{lemma}\label{l:uniformity}
	Let $\mathcal{G}$ be a compactly generated pseudogroup. If there is a metric $d$ and a generating pseudo{\textasteriskcentered}group $\mathcal{S}$ such that every point is a point of $(\mathcal{S},d)$-equicontinuity, then $\mathcal{G}$ is equicontinuous.
\end{lemma}
\begin{proof}
	Let $(U,F,\widetilde F)$ be a system of compact generation with $\widetilde F\subset \mathcal S$ (Lemma~\ref{l:systemcg}). 
 By Proposition~\ref{p:halo},  for every $x\in U$ there is $\rho_x>0$ such that $U_x:=B(x,\rho_x)\subset U$ and
	\begin{equation*}
 U_x\subset\dom\tilde f\qquad \text{for every}\ f\in \langle F\rangle\ \text{with}\ \dom f\cap U_x\neq \emptyset.
	\end{equation*}
	The sets $U_x$, $x\in U$, form an open cover of $U$. 
    Since $\mathcal{G}|_U$ is equivalent to $\mathcal{G}$, it must also be compactly generated, so choose some relatively compact open subset $V\subset U$  meeting every orbit, and choose a finite family $\{U_x\}_{x\in I}$ covering $V$.
    Shrink every $U_x$ to obtain another finite covering $\{V_x\}_{x\in I}$ of $V$ and such that $\overline{V_x}\subset U_x$ for every $x\in I$.

    Let $\epsilon>0$ and cover each set $\overline{V_x}$ $(x\in I)$ by finitely many balls $B(z_{x,j},\delta_{z_{x,j}}(\epsilon/2))$, where $z_{x,j}\in U_x$ and $\delta_{z_{x,j}}$ is a modulus of $(\mathcal S,d)$-equicontinuity at $z_{x,j}$. Let $\delta'(\epsilon)$ be a Lebesgue number for all these coverings; that is, if $u,v$ are contained in some $V_x$ for some $x\in I$ and $d(u,v)<\delta'(\epsilon)$, then there is some $j$ such that $u,v\in B(z_{x,j},\delta_{z_{x,j}}(\epsilon/2))$.

    Let us now show that $\mathcal{G}|_{V}$ is equicontinuous with respect to the cover by quasi-local metric spaces ${(V_x,d)}_{x\in I}$ (see Definition~\ref{d:qlm}). Choose $\langle F\rangle|_{V}$ as a generating pseudo$\ast$group for $\mathcal{G}|_{V}$, let $x,y\in I$, $u,v\in V_x$ and $f\in\langle F\rangle$ be such that $fu,fv\in V_y$, then we need to show that 
    \[
    d(u,v)<\delta'(\epsilon)\quad\Longrightarrow\quad d(fu,fv)<\epsilon\quad \text{for every}\ \epsilon>0.
    \]
    By our choice of $\delta'(\epsilon)$, there is some $z:=z_{x,j}$ so that $u,v\in B(z,\delta_{z}(\epsilon/2))$. Moreover, we have already established that $U_x\subset\dom \tilde f$, so $\tilde fz$ is well-defined. Now we get
    \[
    d(f u,\tilde fz),\ d(f v, \tilde f z )<\epsilon/2,
    \]
    so $d(fu,fv)< \epsilon $ by the triangle inequality.
 
	We have proved that $\mathcal{G}|_V$ is equicontinuous with respect to the cover by quasi-local metric spaces $\{(U_{x_i},d_{x_i})\}$. Since $\mathcal{G}|_V$ is equivalent to $\mathcal{G}$, the result follows by Proposition~\ref	{p:equicontinuityinvariant}.
\end{proof}

\begin{proof}[Proof of Thm.~\ref{t:sicactionpseudogroup}]
	Let us prove that, if the action $G\acts X$ is sensitive, then the induced pseudogroup $\mathcal{G}$ is sensitive too, the converse implication being trivial. Let $c_G>0$ be a sensitivity constant for $G\acts X$, let $H=\{f_1,\ldots,f_n\}\subset G$ be a symmetric  finite generating set (in the group-theoretic sense), let $d$ be a metric on $X$, and let $\mathcal{S}$ be a generating pseudo{\textasteriskcentered}group for $\mathcal{G}$. Since $H\subset \mathcal{G}$ and $\mathcal{S}$ generates $\mathcal{G}$, Lemma~\ref{l:generate} yields a finite sequence of open coverings of $X$, 
	\[
	\widetilde{\mathcal{U}}_i=\{\widetilde{U}_{i,j}\},\qquad i=1,\ldots,n,
	\] such that 
	\[\tilde f_{i,j}:=(f_i)|_{\widetilde U_{i,j}}\in\mathcal{S}\qquad \text{for every} \ i,j;\]  furthermore, we may assume that every $\widetilde{\mathcal{U}}_i$ is finite because $X$ is compact. Let $\mathcal{U}_i=\{U_{i,j}\}$ be a shrinking of $\widetilde{\mathcal{U}}_i$,  let $ f_{i,j}:=(f_i)|_{U_{i,j}}$, and let
	\[
	F=\{f_{i,j}\},\qquad \widetilde F=\{\tilde f_{i,j}\}.
	\]
	Then $(X,F,\widetilde F)$ is a system of compact generation for $\mathcal{G}$.
	
	Let us show that $(S,d)$ is sensitive with constant $c_F:=\min\{\sigma/2,c_G\}$, where $\sigma:=\sigma(X,F,\widetilde F)$ is given by Proposition~\ref{p:halo} (note that its value does not depend on $x$).
	If $(X,F,\widetilde F)$ satisfies \ref{p:halo}(\ref{i:nothalo}) at every point in $X$, then we are done, so suppose that \ref{p:halo}(\ref{i:halo}) holds for some $x\in X$ and  $\rho>0$. Since $G$ is sensitive, there are \[g=f_{i_k}\cdots f_{i_1}\in G\] and $y\in B(x,\rho)$ with $d(g(x),g(y))\geq c_G$. Clearly, there is a sequence $j_1,\ldots, j_k$ such that $x\in\dom h$, where $h=f_{i_k,j_k}\cdots f_{i_1,j_1}$. But $B(x,\rho)\subset\dom \tilde h$ by \ref{p:halo}(\ref{i:halo}),  whence $y\in \dom \tilde h$ and
	\[
	d(\tilde h(x),\tilde h(y))\geq c_G\geq c_F.
	\]
	This shows that $\mathcal{G}$ is $(\mathcal{S},d)$-sensitive and, since both $\mathcal{S}$ and $d$ were arbitrary, the result follows.
\end{proof}

\subsection{Sensitive group actions whose induced pseudogroups are not sensitive}\label{s:nonsensitiveaction}

In this section we construct the counterexamples of Theorem~\ref{t:linked}. Let us start by defining a family of linked twists on the $2$-torus $\mathbb{T}^2$:
Let $p_z$ ($z\in\mathbb{Z}$) denote the following integer-indexed sequence of real numbers:
\[
p_z=\begin{cases}
	1-2^{-1-z} \quad &\text{if}\ z\geq 1,\\
	2^{z-2} \quad &\text{if}\ z\leq 0.
\end{cases}
\]
Let 
\[H=\{(x,y)\in \mathbb{T}^2\mid 1/4\leq y\leq 3/4\},\]
and  
\[
V_z=\{(x,y)\in \mathbb{T}^2\mid p_z\leq x\leq p_{z+1}\}\qquad (z\in \mathbb{Z}).
\] 
Let $T_h\colon \mathbb{T}^2\to \mathbb{T}^2$ be the horizontal twist defined by
\[
T_h(x,y)=\begin{cases}
	(x+2(y-\frac{1}{4}),y)\quad &\text{if}\ (x,y)\in H,\\
	(x,y)\quad &\text{else};
\end{cases}
\]
and, for $m\in\mathbb{N}$, let $T_{v,m}\colon\mathbb{T}^2\to \mathbb{T}^2$ be  the vertical twist:
\[
T_{v,m}(x,y)=\begin{cases}
	(x,y+2^{2+|z|}(x-p_z))\quad &\text{if}\ (x,y)\in V_z,\ |z|\leq m,\\
	(x,y)\quad &\text{else}.
\end{cases}
\]
Letting $T_m=T_{v,m}\circ T_h$ we obtain a  sequence of linked twist maps on $\mathbb{T}^2$.

By Theorem~\ref{t:twisttt}, $T_z$ is topologically transitive on 
\[
M_z= H\cup \bigcup_{|z|\leq m} V_z
\]
and is the identity on $\mathbb{T}^2\setminus M_m$ for every $m\in\mathbb{N}$. Note that, by the definition of the sequence $p_z$, we have
\begin{equation}\label{unionmm}
	\bigcup_{m\geq 0} M_m=\mathbb{T}^2 \setminus\big(\{0\}\times([0,1/4)\cup (3/4,1])\big).
\end{equation}
Moreover, $T_m$ is affine on $\mathbb{T}^2\setminus \Delta_m$, where
\[
\Delta_m:= \partial H\cup \bigcup_{|z|\leq m}T_h^{-1}(\partial V_z).
\]

\begin{lemma}
	Let $(p/q,r/s)\in \mathbb{T}^2$ be a point with rational coordinates. Then, for every $m\geq 0$, its $T_m$-orbit is contained in the finite set 
	\[
	\left\{\left(\frac{l_1}{d},\frac{l_2}{d}\right)\mid l_1,l_2\in \{0,\ldots,d-1\}\right\},
	\]
	where $d=\lcm(q,s,2^{m+2})$.
\end{lemma}
\begin{proof}
	Follows from the definition of $T_h$ and $T_{v,m}$.
\end{proof}

We are now in position to introduce our examples. We begin by showing that an action $G\acts X$ with $G$ finitely generated but $X$ non-compact might be sensitive, while the induced pseudogroups is not. Let $X=\mathbb{T}^2\times \mathbb{Z}$ and let
\[
\sigma((x,y),z)= (T_{|z|}(x,y), z),\qquad \tau((x,y),z)=((x,y),z+1).
\]

\begin{proposition}\label{p:gmathcalg}
	The subgroup of $\Homeo(X)$ generated by $\sigma$ and $\tau$ is sensitive as a group action. The pseudogroup generated by $\sigma$ and $\tau$, however, is not sensitive to initial conditions (in the sense of Definition~\ref{d:sensitivity}).
\end{proposition} 
\begin{proof}
	We start by proving that the induced group action is sensitive. Let $d$ be a compatible metric on $X$ that restricts to the standard flat metric $d_z$ on every  $\mathbb{T}^2\times\{z\}\cong \mathbb{T}^2$. Choose  $c>0$ such that 
	\begin{itemize}
		\item the action of $T_0$ on $M_0$ is $c$-sensitive (with respect to $d_0$),  and
		\item $c<1/8$.
	\end{itemize} 
	Let $((x,y),z)\in X$, and let $U\times\{z\}$ be a neighborhood of $((x,y),z)$. By~\eqref{unionmm}, $\bigcup_z M_{|z|}$ is dense in $\mathbb{T}^2$, so there is $n>0$ such that 
	\[\tau^n(U\times\{z\})\cap \left(M_{n+z}\times\{n+z\}\right)\neq \emptyset.\]
	Since $T_{n+z}$ is topologically transitive on $M_{n+z}$, there is $m>0$ so that 
	\[
	\sigma^m\tau^n(U\times\{z\})\cap (M_{0}\times\{n+z\})\neq \emptyset,
	\]
	and therefore
	\[
	\tau^{-n-z}\sigma^m\tau^n(U\times\{z\})\cap (M_{0}\times\{0\})\neq \emptyset.
	\]
	Let $\phi=\tau^{-n-z}\sigma^m\tau^n$ for the sake of simplicity. Since the action of $T_0$ is $c$-sensitive on $M_0$,  if $\phi((x,y),z)\in M_0\times\{0\}$, then there are $l>0$ and $((u,v),z)\in U\times\{z\}$ such that 
	\[
	d(\sigma^l\phi((x,y),z),\sigma^l\phi((u,v),z))\geq c.
	\]
	If $\phi((x,y),z)\notin M_0\times\{0\}$, then, since
	the action of $T_0$ is topologically transitive on $M_0$, there are $l>0$ and $((u,v),z)\in U\times\{z\}$ such that 
	\[
	\tau^l\phi((u,v),z)\in \left[\frac{3}{8},\frac{5}{8}\right]^2\times\{0\}.
	\]
	But $[1/4,3/4]\subset M_0$ and $\phi((x,y),z)\notin M_0\times\{0\}$, so
	\[
	d(\sigma^l\phi((x,y),z),\sigma^l\phi((u,v),z))\geq\frac{1}{8}\geq c.
	\]
	
	Let us now show that the pseudogroup $\mathcal{G}$ generated by $\sigma$ and $\tau$ is not sensitive. Consider the point $((0,0),0)\in X$. Note that, since $(0,0)\notin M_m$ for every $m\in \mathbb{N}$, $\sigma$ is the identity on a neighborhood of $((0,0),z)$ for every $z\in \mathbb{Z}$. 
	For each $z\in \mathbb{Z}$, let $U_z,$ $O_z$ be  open neighborhoods of $(0,0)$ such that $\overline{O_z}\subset \mathbb{T}^2 \setminus M_{|z|+1}$ and $\overline{U_z}\subset O_z$. Now let
	\[
	U=\bigcup_{z\in\mathbb{Z}} (\mathbb{T}^2\setminus \overline{U}_z)\times \{z\},\qquad O=\bigcup_{z\in\mathbb{Z}} O_z\times \{z\},
	\]
	and consider the pseudo{\textasteriskcentered}group $\mathcal{S}$ generated by
	\[
	F=\{\sigma|_U, \sigma|_O, \tau|_U, \tau|_O\}.
	\]
    Clearly, $\mathcal{S}$ generates $\mathcal{G}$.
	Since 
	\[
	\mathcal{G}((0,0),0)=\{((0,0),z)\mid z\in\mathbb{Z}\},
	\] 
	the only maps in $F$ that are defined on the orbit of $((0,0),0)$ are $\tau|_O$ and $\sigma|_O$, which are isometries (recall that $\sigma|_O=\id|_O$), so we have that every map in $\mathcal{S}$ defined on $((0,0),0)$ is an isometry with respect to $d$, and therefore $\mathcal{G}$ is not sensitive to initial conditions.
\end{proof}

We have constructed the first counterexample of Theorem~\ref{t:linked}.
We can repurpose this machinery to obtain the second counterexample: an action $F_\omega\acts \mathbb{T}^2$ that does not satisfy Theorem~\ref{t:sicactionpseudogroup}, where $F_\omega$ is the free group with countably many generators. Define the action by mapping a sequence freely generating  $F_\omega$ to the sequence  $T_m$, $m\geq 0$. The proof that $F_\omega\acts \mathbb{T}^2$ is sensitive and the pseudogroup is not sensitive is virtually identical to Proposition~\ref{p:gmathcalg}, so we leave the details to the reader.

\subsection{Main results}\label{s:main}

In this section we will prove Theorems~\ref{t:invariant},~\ref{t:auslanderyorkepseudo}, and~\ref{t:ttdposicpseudo}, in that order.
We begin with the following preliminary result, which follows arguments from Lemma~8.8 and Theorem 15.1 in~\cite{AlvarezCandel}.

\begin{proposition}\label{p:invariant}
	Let $\mathcal{G}$ act on a locally compact and separable metric space $(X,d)$, let $\mathcal{S}\subset \mathcal{G}$ be a generating pseudo{\textasteriskcentered}group, and let $\Phi\colon (X,\mathcal{G})\to (Y,\mathcal{H})$ be an equivalence. Then there is a generating pseudo{\textasteriskcentered}group $\mathcal{T}$ for $\mathcal{H}$ and a metric $d'$ on $Y$ satisfying the following condition:  if there are $x\in X$, $\epsilon,\delta>0$  such that
	\[
	d(x,u)<\delta\quad\Longrightarrow\quad d(sx,su)<\epsilon
	\]  
	for every $u\in X$ and $s\in \mathcal{S}$ with $x,u\in\dom s$,
	then, for every $y\in\Phi(x)$ there is $\delta_y>0$ such that
	\[
	d'(y,v)<\delta_y\quad\Longrightarrow\quad d'(ty,tv)<\epsilon
	\]
	for every  $t\in \mathcal{T}$ with $y,v\in\dom t$.
\end{proposition}

\begin{proof}
	We begin by proving the following preliminary result.
	\begin{claim}\label{c:phizero}
		There is a subset $\widehat{\Phi}_0\subset \Phi$ such that
		\begin{enumerate}[(a)]
			\item $\dom \phi$ and $\im \phi$ are relatively compact for every $\phi\in\widehat{\Phi}_0$,
			\item \label{i:phizeros} the map $\psi^{-1}\phi$ belongs to $\mathcal{S}$ for every $\phi,\psi\in\widehat{\Phi}_0$, and
			\item \label{i:phizerox}	$\{\im\phi\mid \phi\in\widehat{\Phi}_0\}$ is a locally finite open covering of $Y$.
		\end{enumerate} 
	\end{claim}
	
	First note that, since $Y$ is a locally compact and separable metric space and $\Phi$ is an equivalence, we can find a sequence, $\phi_1, \phi_2, \ldots$, in $\Phi$ such that
	\begin{itemize}
		\item every $\phi_n$ has an extension $\tilde\phi_n\in\Phi$ with $\overline{\dom\phi_n}\subset\dom\tilde\phi_n$,
		\item  $\dom \tilde \phi_n$ and $\im\tilde \phi_n$ are relatively compact for every $n\geq 1$, and
		\item  $\{\im \phi_n\mid n\geq 1\}$ and $\{\im \tilde \phi_n\mid n\geq 1\}$ are locally finite open coverings of $Y$.
	\end{itemize} 
	
	We  define now by induction on $n$ an increasing sequence of finite subsets $\widehat{\Phi}_{0,n}\subset\Phi$ ($n\geq 1$)
	 so that $\psi^{-1}\phi$ belongs to $\mathcal{S}$ for all $\phi,\psi\in\widehat{\Phi}_{0,n}$ and
	\[
	\im\phi_1\cup\cdots\cup\im\phi_n\subset\bigcup_{\phi\in\widehat{\Phi}_{0,n}}\im\phi.
	\]
	Let $\widehat\Phi_{0,1}=\{\phi_1\}$ and, for $n>1$, assume that we have defined $\widehat\Phi_{0,n-1}$ satisfying the required properties. Lemma~\ref{l:generate} yields a finite open covering $\{U_i\}_{i\in I}$ of $\overline{\dom\phi_n}$ such that every $U_i$ is relatively compact and the restriction of  $\psi^{-1}\tilde\phi_n$ to every $U_i$ belongs to $\mathcal{S}$ for every $\psi\in\widehat\Phi_{0,n-1}$. Letting
	\[
	\widehat\Phi_{0,n}=\widehat\Phi_{0,n-1}\cup\{\,\tilde{\phi}_{n|U_i}\mid i\in I\,\},
	\]
	we get
	\[
	\im\phi_n\subset\bigcup_{i\in I} \tilde\phi_n(U_i)=	\bigcup_{i\in I}\im(\tilde{\phi}_{n}|_{U_i}).
	\]
	The induction hypothesis now yields
	\[
	\im\phi_1\cup\cdots\cup\im\phi_n\subset\bigcup_{\phi\in\widehat\Phi_{0,n}}\im\phi.
	\] 
	
	Let us show by cases that $\psi^{-1}\phi$ belongs to $\mathcal{S}$ for all $\phi, \psi\in\widehat\Phi_{0,n}$.
	If $\phi, \psi\in\widehat\Phi_{0,n-1}$, then it follows from the induction hypothesis, so assume first that $\phi=\tilde\phi_{n|U_i}$ for some $i\in I$. If $\psi$ is also of the form $\tilde\phi_{n|U_j}$ for some $j\in I$, then $\psi^{-1}\phi$ is the identity on its domain, so $\psi^{-1}\phi\in \mathcal{S}$. 
	If, on the other hand, $\psi\in\widehat\Phi_{0,n-1}$, then  $\psi^{-1}\phi=\psi^{-1}\tilde\phi_{n|U_i}\in \mathcal{S}$ by the definition of $U_i$. The only case remaining is when $\phi\in\widehat\Phi_{0,n-1}$ and $\psi=\tilde\phi_{n|U_i}$ for some $i\in I$, which follows from the previous argument because $\mathcal{S}$ is symmetric. This completes the proof of Claim~\ref{c:phizero} by taking $\widehat\Phi_0=\bigcup_n\widehat\Phi_{0,n}$.

	We turn to the task of defining $\Phi_0$ and the generating pseudo{\textasteriskcentered}group $\mathcal T$. 
	Let $\{V_{\hat\phi}\mid\hat\phi\in\widehat\Phi_0\}$ be a shrinking of the covering $\{\im\hat\phi\mid\hat\phi\in\widehat\Phi_0\}$; that is, 
	$\bigcup_{\hat\phi\in\Phi_0} V_{\hat\phi}=Y$ and
	$V_{\hat\phi}$ is an open set satisfying $\overline{V_{\hat\phi}}\subset \im \hat\phi$
	for every $\hat\phi\in\widehat\Phi_0$. 
	\begin{claim}\label{c:p}
		For every $x\in X$, there is a open neighborhood $U_x$ of $x$ such that
		\[
		U_x\cap V_{\hat\phi}\neq\emptyset\quad\Longrightarrow\quad U_x\subset \im\hat\phi
		\]
		for every $\hat\phi\in\widehat\Phi_0$.
	\end{claim}
	Since $\{\im\hat\phi\mid\hat\phi\in\widehat\Phi_0\}$ is locally finite by Claim~\ref{c:phizero}(\ref{i:phizerox}), 
	\[
	U_x=\Big(\bigcap_{\hat\phi\in\widehat\Phi_0,\ x\in\im\hat\phi}\im\hat\phi\Big)\setminus\Big(\bigcup_{\hat\phi\in\widehat\Phi_0,\  x\notin\overline{V_{\hat\phi}}}\overline{V_{\hat\phi}}\Big)
	\] is an open set that 	satisfies the required properties, proving Claim~\ref{c:p}.

	
	For every $\hat\phi\in\widehat\Phi_0$, let $\{P_i\mid i\in I_{\hat\phi}\}$ be a locally finite open covering of  the open set $V_{\hat\phi}$ such that every $P_i$ satisfies
	\begin{align}\label{imv}
		P_i\cap V_{\hat\psi}\neq\emptyset\quad &\Longrightarrow\quad P_i\subset\im\hat\psi\qquad \forall \hat\psi\in\widehat\Phi_0,\\
		\label{diamhatphi}\diam \hat\phi^{-1}(P_i)&<d(\hat\phi^{-1}(P_i),\dom\hat\phi\setminus \hat\phi^{-1}(V_{\hat\phi})).
	\end{align} 
	From now on, given $\hat\phi\in \widehat\Phi_0$ and $i\in I_{\hat\phi}$, let $\phi_i$ be shorthand for $\hat\phi|_{P_{i}}$. Let
	\begin{align*}
		\Phi_0&=\{\,\phi_i\mid \hat\phi\in\widehat\Phi_0,\ i\in I_{\hat \phi}\,\},\\
    		\mathcal T&=\{\, \psi_j s \phi_i^{-1}\mid \phi_i,\psi_j\in\Phi_0,\ s\in \mathcal{S}\,\}\cup\{\id|_U : U \text{ open in Y}\}.
	\end{align*}
	It is elementary to check that $\mathcal{T}$ is a pseudo{\textasteriskcentered}group, so let us prove that $\mathcal{T}$ generates $\mathcal{H}$. Let $h\in \mathcal{H}$ and let $x\in\dom h$. By the definitions of $\{V_{\hat\phi}\mid \phi\in\widehat\Phi_0\}$ and $\{P_i\mid i\in I_{\hat\phi}\}$, the collection 
	\[
	\{\,P_i\mid i\in I_{\hat\phi},\ \phi\in\widehat{\Phi}_0\,\}=\{\,\im \phi_i\mid i\in I_{\hat\phi},\ \hat\phi\in\widehat{\Phi}_0\,\}
	\]
	is an open covering of $Y$. Therefore, there are $\phi_i$, $\psi_j\in\Phi_0$ so that $x\in\im\phi_i$, $f(x)\in\im\psi_j$. Thus $\psi^{-1}_jh\phi_i\in\mathcal{G}$ by Definition~\ref{d:equiv}(\ref{i:equivgamma}). Since $\mathcal{S}$ generates $\mathcal{G}$, there must be an open neighbourhood $U$ of $\phi^{-1}_ix$ so that the restriction of $\psi^{-1}_jh\phi_i$ to $U$ belongs to $\mathcal{S}$ by Lemma~\ref{l:generate}. Then $h$ coincides with $\psi_j s\phi_i^{-1}\in \mathcal{S}$ over $\phi_i(U)$. We have proved that, for every $h\in \mathcal{H}$ and $x\in\dom h$, the restriction of $h$ to some open neighbourhood of $x$ belongs to $\mathcal{T}$, so $\mathcal{T}$ generates $\mathcal{H}$ by Lemma~\ref{l:generate}.

	We now prove some preliminary results needed to define the metric $d'$. For each $\hat\phi\in\widehat{\Phi}_0$, let $D_{\hat\phi}\colon \im\hat\phi\times\im\hat\phi\to \mathbb{R}_{\geq 0}$ be the  metric defined on the open set $\im\hat\phi$ by  $D_{\hat\phi}(x,y)=d(\hat\phi^{-1}x,\hat\phi^{-1}y)$.
	If $u,v\in\im\hat\phi$ for some $\hat\phi\in\widehat{\Phi}_0$, let 
	\begin{equation*}
	\overline D(u,v)=\sup \{\,D_{\hat\phi}(u,v)\mid \hat\phi\in \widehat{\Phi}_0,\ u,v\in\im\hat\phi\,\}.
	\end{equation*}
	A pair $(u,v)\in Y\times Y$ is \emph{admissible} if there is $\hat\phi\in\widehat{\Phi}_0$ such that $u,v\in V_{\hat\phi}$ and
	\[
	\{u,v\}\cap V_{\hat\psi}\neq\emptyset\quad\Longrightarrow\quad\{u,v\}\subset \im\hat\psi\qquad \forall\hat\psi\in\widehat{\Phi}_0.
	\]
	Let $S_{u,v}$ be the set of sequences $(z_0,\ldots,z_n)$ of arbitrary finite length with $z_0=u$, $z_n=v$, and such that $(z_{i-1},z_{i})$ is an admissible pair for every $i=1,\ldots,n$. The following properties are elementary:
	\begin{align}\label{suvref}
		(u,u)&\in S_{u,u},\\
		(z_0,\ldots,z_n)\in S_{u,v}&\Longrightarrow (z_n,\ldots,z_0)\in S_{v,u},\label{suvsym}\\
		\label{suvtrans}
		\begin{rcases*}
			(z_0,\ldots,z_m)\in S_{u,v}\\ 
			(z_m,\ldots,z_{m+n})\in S_{v,w}
		\end{rcases*} &\Longrightarrow (z_0,\ldots,z_{m+n})\in S_{u,w}.
	\end{align}
	Set  
	\begin{equation}\label{definitiondprime}
		d'(u,v)=\begin{cases}\infty \qquad &\text{if}\quad S_{u,v}=\emptyset,\\ \inf_{(z_0,\ldots,z_n)\in S_{u,v}}\sum_{k=1}^n\overline{D}(z_{k-1},z_k) \qquad &\text{if}\quad S_{u,v}\neq\emptyset.\end{cases}
	\end{equation} It follows from~\eqref{suvref}--\eqref{suvtrans} that $d'$ is a pseudometric in $Y$. To prove that it is actually a metric, we need the following result.
	
	\begin{claim}\label{c:vimphi}
		Let $\hat\phi\in\widehat\Phi_0$, let $u\in V_{\hat\phi}$, and let $v\in Y$ be such that $S_{u,v}\neq\emptyset$. Then
		\[
		d'(u,v)\geq\begin{cases}
			\min\{D_{\hat\phi}(u,v), D_{\hat\phi}(u,\im\hat\phi\setminus V_{\hat\phi})\} &\text{if}\ v\in V_{\hat\phi},\\
			D_{\hat\phi}(u,\im\hat\phi\setminus V_{\hat\phi}) &\text{if}\ v\notin V_{\hat\phi}.
		\end{cases}
		\]
	\end{claim}
	Let $(z_0,\ldots,z_n)\in S_{u,v}$. Suppose first that  $\{z_{i-1},z_i\}\subset V_{\hat\phi}$ for every $i=1,\ldots,n$, then
	\[
	\sum_{k=1}^n\overline{D}(z_{k-1},z_k)\geq \sum_{k=1}^n D_{\hat\phi}(z_{k-1},z_k)\geq D_{\hat\phi}(z_0,z_n)=D_{\hat\phi}(u,v)
	\]
	by the triangle inequality. Assume now that  $m$ is the first index in  $\{1,\ldots n\}$ satisfying $z_{m}\notin V_{\hat \phi}$. Since $(z_{m-1},z_{m})$ is an admissible pair and $z_{m-1}\in \im\phi_i\subset V_{\hat\phi}$, we get $z_{m}\in \im\hat\phi$. Therefore
	\[
	\sum_{k=1}^n\overline{D}(z_{k-1},z_k)\geq \sum_{k=1}^{m} D_{\hat\phi}(z_{k-1},z_k)\geq D_{\hat\phi}(z_0,z_{m})\geq D_{\hat\phi}(u,\im\hat\phi\setminus V_{\hat\phi});
	\]
	this completes the proof of Claim~\ref{c:vimphi}.
	
	\begin{claim}\label{c:compatible}
		$d'$ is a compatible metric on $Y$.
	\end{claim}
	Let us prove that $d'$ is a metric: Let $u,v\in Y$ be such that $d'(u,v)=0$, so $S_{u,v}\neq\emptyset$. Take any $\hat\phi\in\widehat{\Phi}_0$ such that $u\in V_{\hat\phi}$. Since 
	\[
	D_{\hat\phi}(u,\im\hat\phi\setminus V_{\hat\phi})>0, 
	\]
	it follows from Claim~\ref{c:vimphi} that $v\in V_{\hat\phi}$ and $D_{\hat\phi}(u,v)\leq d'(u,v)=0$. But $D_{\hat\phi}$ is a metric on $\im\hat\phi$, so $u=v$ as desired.
	
	Let us show that $d'$ is a compatible metric. We start by showing that every neighborhood in $X$ contains a $d'$-ball, so let $U$ be a neighborhood of $x$. Since $\{\im\hat\phi\mid\hat\phi\in\widehat{\Phi}_0\}$ is a locally finite cover, we may assume 
	\[
	\{\hat\phi\in\widehat{\Phi}_0\mid x\in V_{\hat\phi}\}=\{\hat\phi_1,\ldots,\hat\phi_n\}
	\] 
	for some $n\in\mathbb{N}$.
	The metrics $D_{\hat\phi_i}$ are compatible over $\im\hat\phi_i$, so we can find some $r>0$ satisfying 
	\begin{align*}
		B_{D_{\hat\phi_i}}(x,r)\subset U,\qquad
		d(B_{D_{\hat\phi_i}}(x,r),\im\hat\phi_i\setminus V_{\hat\phi_i})>r
	\end{align*}
	for every $i=1,\ldots,n$; then, for every $y\in B_{d'}(x,r)$,
	\[
	r>d'(x,y)\geq D_{\hat\phi_i}(x,y)
	\]
	 by Claim~\ref{c:vimphi}, so $y\in B_{D_{\hat\phi_i}}(x,r)$ and hence $y\in U$.
	
	Consider now a ball $B_{d'}(x,r)$.  Choose a neighborhood $U$ of $x$ small enough so that 
	\begin{align*}
	U\subset V_{\hat\phi_i},\qquad
	U\subset B_{D_{\hat\phi_i}}(x,r)
	\end{align*}
	for $i=1,\ldots,n$.
	This means that $(x,u)\in S_{x,u}$ for every $u\in U$, so 
	\[
	d'(x,u)\leq\overline{D}(x,u)=\sup_i D_{\hat\phi_i}(x,u)<r.
	\]
	Hence $U\subset B_{d'}(x,r)$, proving Claim~\ref{c:compatible}.

	Let  $\epsilon,\delta>0$,  $x\in X$ and $y\in \Phi(x)$ be as in the statement of the theorem. 
	By definition of $\mathcal{T}$, every map that is not the identity is of the form $\psi_j s\phi^{-1}_i$, where $s\in \mathcal{S}$ and $\psi_j,\phi_i\in\Phi_0$ for some $\hat\psi,\hat\chi\in\widehat{\Phi}_0$. Recall that the notation $\psi_j$ means that $\psi_j$ is of the form $\hat\psi|_{P_j}$, where $\{P_j\mid j\in I_{\hat\psi}\}$ is a locally finite covering of $V_{\hat\psi}$. Since the covering $\{V_{\hat\chi}\mid \hat\chi\in\widehat{\Phi}_0\}$ was also locally finite, there are only finitely many points in $X$ that are sent to $y$ by a map in $\Phi_0$; denote them by $x_1,\ldots, x_n$. Moreover, since $\Phi$ is an equivalence, all these points lie on the same orbit, so we can find a uniform  $\delta_y$ satisfying
 	\begin{equation}\label{eq:deltay}
	d(x_l,u)<\delta_y\quad\Longrightarrow\quad d(sx_l,su)<\epsilon
	\end{equation}
	for every $l=1,\ldots,n$, $u\in X$ and $s\in \mathcal{S}$ with $x_l,u\in\dom s$. 
 
  Let $t\in \mathcal \mathcal{T}$ satisfy $y\in \dom t$, and let $v\in \dom t\cap B_{d'}(y,\delta_y)$. Let $t=\psi_j s\phi^{-1}_i$, then, in particular, \eqref{eq:deltay} is also satisfied for $x_l=\phi_i^{-1}(y)$, which belongs to the set $x_1,\ldots,x_n$.

 Now $y,v\in\dom t$ implies $y,v\in \im\phi_i=P_i\subset V_{\hat\phi}$. Hence $\{y,v\}\in S_{y,v}$ by~\eqref{imv} and 
	\[
	D_{\hat\phi}(y,v)<D_{\hat\phi}(y,\im\hat\phi\setminus V_{\hat\phi})
	\] 
	by~\eqref{diamhatphi}, so Claim~\ref{c:vimphi} yields 
	\[D_{\hat\phi}(y,v)\leq d'(y,v)< \delta_y, \] 
	and therefore 
	\[
	d(x_l,u)<\delta_y
	\]
	by the definition of $D_{\hat\phi}$, where $u=\phi_i^{-1}v$. 
	
	Let  $\chi_k\in \Phi_0$  be such that $ty,tv\in \im  \chi_k$. Then 
	\[
	d(\chi_k^{-1}ty,\chi_k^{-1}tv)=d((\chi_k^{-1}\psi_j)sx_l,(\chi_k^{-1}\psi_j)su).
	\]
	Since $\chi_k^{-1}\psi_j\in \mathcal{S}$ by Claim~\ref{c:phizero}(\ref{i:phizeros}) and $d(x_l,u)<\delta_y$, we have
	\[
	D_{\hat\chi}(ty,tv)=d(\chi_k^{-1}ty,\chi_k^{-1}tv)<\epsilon.
	\]

    The covering $\{\im \hat \eta\mid \hat\eta\in \widehat{\Phi}\}$ is locally finite, so again there are only finitely many maps such that $ty$ and $tv$ are contained in their image. In particular, this implies that $\overline{D}(ty,tv)< \epsilon$ since we are taking the maximum of a finite set of values smaller than $\epsilon$.
	Both $sx$ and $su$ belong to $\dom\psi_j\subset V_{\hat\psi}$, so~\eqref{definitiondprime} yields
	\[
	d'(ty,tv)\leq\overline D(ty,tv)<\epsilon.\qedhere
	\] 
\end{proof}

\begin{corollary}\label{c:sensitive}
	Being sensitive to initial conditions is invariant by equivalences of pseudogroups acting on locally compact spaces.
\end{corollary}
\begin{proof}
	Suppose that $\mathcal{G}$ is not sensitive; then there is a metric $d$ and a generating pseudo{\textasteriskcentered}group  $\mathcal{S}$ such that, for every $\epsilon>0$, there are $x_\epsilon$ and $\delta_\epsilon$ with
	\[
	d(x_\epsilon,u)<\delta_\epsilon\quad\Longrightarrow\quad d(s(x_\epsilon),s(u))<\epsilon
	\]
	for all $u\in X$ and $s\in\mathcal{S}$ with $x_\epsilon,u\in\dom s$.
	
	Let $(Y,\mathcal{H})$ be a pseudogroup, and let $\Phi\colon (X,\mathcal{G})\to (Y,\mathcal{H})$ be an equivalence.  Proposition~\ref{p:invariant} yields a generating pseudo{\textasteriskcentered}group $\mathcal{T}$ for $\mathcal H$ and a metric $d'$ on $Y$. 
Letting $y_\epsilon\in\Phi(x_\epsilon)$, Proposition~\ref{p:invariant} yields
	\[
	d'(y_\epsilon,v)<\delta_{\epsilon,y_\epsilon}\quad\Longrightarrow\quad d'(t(y_\epsilon),t(v))<\epsilon
	\]
	for every $\epsilon>0$, $v\in Y$ and $t\in \mathcal{T}$ with $y_\epsilon,v\in \dom t$; this shows that $\mathcal{H}$ is not sensitive to initial conditions.
	
	We have shown that if $\mathcal{G}$ is not sensitive, then neither is $\mathcal{H}$; the result now follows from the symmetry of the equivalence relation for pseudogroups.
\end{proof}

\begin{corollary}\label{c:equicontequivalence}
	Let $\Phi\colon (X,\mathcal{G})\to (Y,\mathcal{H})$ be an equivalence of pseudogroups acting on locally compact spaces. If $x\in X$ is a point of $( \mathcal S,d)$-equicontinuity for $\mathcal{G}$, then, with the notation of Proposition~\ref{p:invariant}, every $y\in \Phi_0 x$ is a point of $(\mathcal T,d')$-equicontinuity for $\mathcal{H}$. In particular, $\mathcal{G}$ is almost equicontinuous if and only if $\mathcal{H}$ is.
\end{corollary}
\begin{proof}
	The first assertion follows immediately from Preposition~\ref{p:invariant} and the definition of equicontinuous point. Suppose now that the points of $(\mathcal S,d)$-equicontinuity are dense in $X$, then they are also dense in the open set $\bigcup_{\phi\in\Phi_0} \dom \phi$. Since $\Phi_0$ sends points of $(\mathcal S,d)$-equicontinuity to points of $(\mathcal T,d')$-equicontinuity and $\{\im\phi\mid\phi\in\Phi_0\}$ is an open covering of $Y$, the points of $(\mathcal T,d')$-equicontinuity are dense in $Y$.
\end{proof}

Corollaries~\ref{c:sensitive} and~\ref{c:equicontequivalence} together with Lemma~\ref{l:dppinvariant} yield Theorem~\ref{t:invariant}. 

We turn our attention to the Auslander--Yorke dichotomy for pseudogroups, which we will subsequently use to prove Theorem~\ref{t:ttdposicpseudo}.

\begin{proof}[Proof of Theorem~\ref{t:auslanderyorkepseudo}]
	Suppose that $\mathcal{G}$ is not sensitive; thus, there is a metric $d$ and a generating pseudo{\textasteriskcentered}group $\mathcal{S}$ so that, for every $n\in \mathbb{N}$, there are $x_n\in X$ and $\delta_n>0$ satisfying
	\begin{equation}\label{auslanderyorke}
		d(x_n,y)<\delta_n\quad\Longrightarrow\quad d(s(x_n),s(y))<1/n
	\end{equation} for every $y\in X$ and $s\in \mathcal{S}$ with $x_n,y\in\dom s$.
	
	Using Lemma~\ref{l:systemcg}, choose a system of compact generation $(U,F,\widetilde F)$ with $\widetilde F\subset \mathcal{S}$; note that any point in $\mathcal{G}x_n$ still satisfies~\eqref{auslanderyorke}, perhaps with a different $\delta_n$, so we may assume without loss of generality that the sequence $x_n$ is contained in $U$. We also have $1/n<\sigma(U,F,\widetilde{F})$ for $n$ large enough, and now Proposition~\ref{p:halo} yields the existence of a sequence $r_n>0$ such that
	\begin{equation}\label{auslanderyorkeftilde}
		\dom f\cap B(x_n,r_n)\neq \emptyset\quad\Longrightarrow\quad B(x_n,r_n)\subset \dom \tilde f
	\end{equation}
	for every $f\in \langle F\rangle$. We will suppose, by passing to a subsequence if necessary, that every $x_n$ satisfies~\eqref{auslanderyorkeftilde}; moreover,  we will also assume by decreasing $r_n$ that $B(x_n,r_n)\subset U$ and $r_n<\delta_n$.
	
	Note that
	\begin{equation}\label{auslanderyorkediam} 
		\diam f(B(x_n,r_n))<2/n
	\end{equation} for every $f\in F$. Indeed, otherwise there would be some $f\in F$ with 
	\begin{equation*}
	B(x_n,r_n)\subset \dom \tilde f,\qquad\diam \tilde f(B(x_n,r_n))\geq 2/n
	\end{equation*}
	by~\eqref{auslanderyorkeftilde}.
	But then the triangle inequality would yield $d(\tilde f(x_n),\tilde f(y))\geq 1/n$ for some $y\in B(x_n,r_n)\subset B(x_n,\delta_n)$, contradicting~\eqref{auslanderyorke}.
	
	Let  
	\[
	V_n=\bigcup_{f\in\langle F\rangle} f(B(x_n,r_n))\qquad \text{for}\ n\geq 1,
	\] which are clearly  open subsets of $U$. Moreover, topological transitivity implies that every $V_n$ is dense in $U$, so by the Baire Category Theorem  $\bigcap_n V_n$ is  also a dense subset of $U$. 
	
	Let us show that every $x\in \bigcap_{n\geq 1} V_n$ is a point of $( F^*,d)$-equicontinuity. Assume for the sake of contradiction that there is $c>0$ such that, for every $r>0$, there are $f\in F^*$ and $y\in B(x,r)$ such that 
	\[x,y\in\dom f,\qquad d(f(x),f(y))\geq c.\] 
	Choose $m$ large enough so that $2/m<c/2$. Since $x\in \bigcap_{n\geq 1} V_n$, there is some $g\in F^*$ such that $g(x)\in B(x_m,r_m)$. By assumption, there are also $y\in X$ and $f\in F^*$ satisfying 
	\[
	y\in \dom g\cap \dom f,\quad g(y)\in B(x_m,r_m),\quad d(f(x),f(y))\geq c.
	\]
	But then
	\[
	d(fg^{-1}(y'),fg^{-1}(x'))=d(f(x),f(y))\geq c
	\]
	for $x'=g(x)$, $y'=g(y)$,
	and now~\eqref{auslanderyorkeftilde} and the triangle inequality yield
	\[
	\max\{d(\tilde f\tilde g^{-1}(x_m),\tilde f\tilde g^{-1}(x')),d(\tilde f\tilde g^{-1}(x_m),\tilde f\tilde g^{-1}(y'))\}\geq c/2>2/m,
	\]
	contradicting~\eqref{auslanderyorkediam}. 
	
	We have proved that, if $\mathcal{G}\acts X$ is topologically transitive and not sensitive to initial conditions, then there is a metric $d$ on $X$ and a system of compact generation $(U,F,\widetilde F)$ such that the set of points of $(F^*,d)$-equicontinuity is dense in $U$.
	Since $\mathcal{G}$ is equivalent to $\mathcal{G}|_U$ by Lemma~\ref{l:equivid}, Corollary~\ref{c:equicontequivalence}  yields that $\mathcal{G}$ is almost equicontinuous.
	
	If $\mathcal{G}$ is minimal, then it is trivial to check that $\bigcap_{n\geq 1} V_n=U$, so by the previous argument there are $d$ and $(U,F,\widetilde F)$ so that every point in $U$ is a point of $(F^*,d)$-equicontinuity. The result then follows by Lemma~\ref{l:uniformity} and Proposition~\ref{p:equicontinuityinvariant}.
\end{proof}

\begin{proof}[Proof of Theorem~\ref{t:ttdposicpseudo}]
	Let $\mathcal{S}$ be a generating pseudo{\textasteriskcentered}group for $\mathcal{G}$ and let $d$ be a compatible metric.	
	By Theorem~\ref{t:auslanderyorkepseudo} (Auslander--Yorke dichotomy for pseudogroups), it is enough to show that, for every $(\mathcal{S},d)$, the set of points of $(\mathcal{S},d)$-equicontinuity is empty.	
	Let $(U,F,\widetilde F)$ be a system of compact generation for $\mathcal{G}$ satisfying $\widetilde F\subset \mathcal{S}$. 
	
	Suppose for the sake of contradiction that $x\in U$ is a point of $(\mathcal{S},d)$-equicontinuity. Then it must satisfy~\ref{p:halo}(\ref{i:halo}) with some $\rho>0$. 
	By Definition~\ref{d:dpo} and Corollary~\ref{c:dpocg}, there are points \[
	q_1\in B(x,\rho/2),\qquad q_2\in B(x,\rho/2)\setminus \mathcal{G}q_1 
	\]with finite  $\mathcal{G}|_U$-orbits. Letting
	\[
	\epsilon<\frac{1}{2}d(\mathcal{G}|_Uq_1,\mathcal{G}|_Uq_2)	\]
	and using the triangle inequality, we can choose a point $q\in\{q_1,q_2\}$ satisfying 	
	\begin{equation}
		\label{gq}d(x,\mathcal{G}q)> \epsilon.
	\end{equation}
	
	For every $n\geq 1$, let $p_n$ be a point in $B(x,\rho/n)$ with finite $\mathcal{G}|_U$-orbit. Since $\mathcal{G}p_n\cap B(x,\rho)$ is finite, there is a finite set $K_n\subset \langle F\rangle$ satisfying that, for every $y\in \mathcal{G}p_n\cap B(x,\rho)$, there is $k\in K_n$ with $ky=p_n$; moreover, since $y\in B(x,\rho)$, each map $k\in K_n$ may be extended to a map $\tilde k\in \langle \widetilde F\rangle$ with $B(x,\rho)\subset \dom \tilde k$. For each $n$, let $\widetilde K_n$ denote the collection of all such extensions. Since $\widetilde K_n$ is a finite set of maps, there is a neighborhood $W_n$ of $q$ so that $W_n\subset B(x,\rho/2)$ and $\tilde k(W_n)\subset B(\tilde k(q),\epsilon/4)$ for every $\tilde k\in\widetilde K_n$.
	
	$\mathcal{G}|_U$ is topologically transitive, so there are maps $f_n$ and points $v_n\in B(x,\rho/n)$ such that $f_nv_n\in W_n$; again, $B(x,\rho)\subset \dom\tilde f_n$ for all $n$.	
	If $d(\tilde f_nv_n,\tilde f_np_n)\geq \rho/2$ for infinitely many $n$, then the triangle inequality yields 
	\[\max\{d(\tilde f_nx,\tilde f_np_n),d(\tilde f_nx,\tilde f_nv_n)\}\geq \rho/4,\]
	showing that $x$ is not a point of $(\mathcal{S},d)$-equicontinuity, a contradiction. 
	Hence, we may assume 
	$d(\tilde f_nv_n,\tilde f_np_n)< \rho/2$ for $n$ large enough. 
	In particular, since $\tilde f_nv_n\in W_n\subset B(x,\rho/2)$, we have $\tilde f_np_n\in B(x,\rho)$, so there are maps $k_n$ in $K_n$ satisfying $k_nf_n(p_n)=p_n$ and $B(x,\rho)\subset \dom \tilde k_n\tilde f_n$ for $n$ large enough.
	
	Now we have
	\begin{equation}\label{bigineq}
	\max\{d(\tilde k_n\tilde f_n p_n,\tilde k_n\tilde f_n x),d(\tilde k_n\tilde f_n x, \tilde k_n\tilde f_n v_n)\}\geq \epsilon/4
	\end{equation} 
        for $n$ large enough because, otherwise, the triangle inequality and $\tilde k_n\tilde f_n p_n=p_n$ would yield 
	\begin{align*}
		d(x,\tilde k_n q)&\leq d(x,p_n)+d(\tilde k_n\tilde f_n p_n,\tilde k_n\tilde f_n x)+\\
		&\phantom{\leq d(x,} 	+d(\tilde k_n\tilde f_n x, \tilde k_n\tilde f_n v_n)+d(\tilde k_n\tilde f_n v_n,\tilde k_n q )\\
		&<\epsilon/4+\epsilon/4+\epsilon/4+\epsilon/4\\
		&=\epsilon,
	\end{align*}
	contradicting~\eqref{gq}.	
	The inequality in~\eqref{bigineq} and the fact that the sequences $\{v_n\}$ and $\{p_n\}$ converge to $x$ are at odds with the assumption that $x$ was a point of $(\mathcal{S},d)$-equicontinuity. Since $x$ was an arbitrary point, we infer that the set of points of $(\mathcal{S},d)$-equicontinuity is empty, as desired.
\end{proof}

\subsection{A non-compactly generated, countably generated and topologically transitive pseudogroup that is not sensitive}\label{s:cantor}
In this section we construct a counterexample showing that compact generation cannot be dropped in the statement of Theorem~\ref{t:ttdposicpseudo}.

Let $Y=\{0,1\}^\mathbb{Z}$, which is a Cantor set.
We use the greek letters $\alpha,\beta,\ldots$ to denote elements of $Y$ and the notation $\alpha=(\alpha_i)_{i\in \mathbb{Z}}, \beta=(\beta_i)_{i\in \mathbb{Z}}$. Let $\sigma\colon Y\to Y$ be the shift function, defined by
\[
(\sigma(\alpha))_i= (\alpha_{i+1}).
\]
Let $G$ denote the subgroup of $\Homeo(Y)$ generated by $\sigma$, endowed with the obvious action $G\curvearrowright Y$. It is well-known that the action $G\curvearrowright Y$  is topologically transitive and has density of periodic orbits.

The point $\mu:=(...,0,0,0,...)$ is a fixed point of $G$. For $n\geq 0$, let
\[
U_n =\{\, \alpha\in Y \mid \alpha_i=0\ \text{for}\ |i|<n \,\}.
\]
Note that $U_0=Y$ and that $\{U_n\}$ is a system of clopen neighbourhoods for $\mu$. We also have
\begin{equation}\label{sigmasigma-1}
	\sigma(U_{n+1}),\ \sigma^{-1}(U_{n+1})\subset U_{n}\qquad  \text{for all}\  n\geq 0. 
\end{equation}

Let $X=\{\,(n,\alpha)\in\mathbb{N}\times Y\mid \alpha\in U_n\,\}$; that is, $X$ is the disjoint union $\bigsqcup_{i\geq 0}U_i$. Let $f,g\in\Ph(X)$ be defined by
\begin{align}
	f(n,\alpha)&=(n,\sigma(\alpha)),\quad & &\dom f=\{\, (n,\alpha)\in X\mid \sigma(\alpha)\in U_n\setminus U_{n+2}\,\}, \label{defnf}\\
	g(n,\alpha)&=(n+1,\alpha),\quad & &\dom g = \{\,(n,\alpha)\in X\mid \alpha\in U_{n+1}\,\}.
\end{align}
In other words, $f$ is defined for pairs $(n,\alpha)$ with $\alpha_i=0$ for $i=-n,\ldots, n+1$ but there is an index in the segment $[-n-1,n+3]$ where $\alpha$ takes the value $1$.
Note that $\dom f$, $\im f$, $\dom g$, and $\im g$ are clopen subsets of $X$. Finally, let $\mathcal{G}$ be the pseudogroup generated by $f$ and $g$.

\begin{lemma}\label{l:gna}
	For every $(n,\alpha)\in X$, $\mathcal{G}(n,\alpha)=\{(m,\beta)\in X\mid \beta\in G\alpha\}$.
\end{lemma}
\begin{proof}
	If follows trivially from the definitions of $f$ and $g$ that every point in $\mathcal{G}(n,a)$ is of the form $(m,\beta)\in X$ for some $\beta\in G\alpha$, so let us prove the reverse inclusion. First note that, since $\mu$ is a fixed point of $G$, the lemma is trivial for $\alpha=\mu$, so assume $\alpha\neq \mu$;  clearly, it is enough to show that 
	\begin{equation}\label{magna}
		\{(m,\sigma (\alpha))\in X\}\subset \mathcal{G}(n,\alpha),
	\end{equation}
	so let's prove it. Let $(m,\alpha)\in X$, we will show first that
	\begin{equation}\label{magna2}
		\text{there is}\ k\in\mathbb{N}\quad  \text{such that}\quad (k,\alpha), (k,\sigma(\alpha))\in X.
	\end{equation} 
	Since $\alpha\neq \mu$, there is a largest $l\geq 0$ such that $\sigma (\alpha)\in U_{l}$. If $l=0$, then $(0,\sigma(\alpha))\in U_0\setminus  U_2$, so  $(0,\alpha)\in \dom f$ by~\eqref{defnf}; if $l\geq 1$, then $\alpha\in U_{l-1}$ by~\eqref{sigmasigma-1}. We chose $l$ so that $\sigma(\alpha)\notin U_{l+1}$, $(l-1,\alpha)\in \dom f$ and $f(l-1,\alpha)=(l-1,\sigma(\alpha))$ by~\eqref{defnf}, proving~\eqref{magna2} and yielding
	\[
	(m,\sigma(\alpha))=g^{m-k}fg^{k-n}(n,\alpha).
	\]
	This shows~\eqref{magna}.
\end{proof}

\begin{corollary}
	$\mathcal{G}$ is topologically transitive and has density of periodic orbits.
\end{corollary}
\begin{proof}
	Let us prove transitivity first. By Lemma~\ref{l:pt}, it is enough to show that there is a dense orbit; choose $\alpha\in Y$ with a dense $G$-orbit, then $\mathcal{G}(0,\alpha)$ is dense in $X$ by Lemma~\ref{l:gna}.
	
	In order to prove density of periodic orbits, let $(m,\beta)\in X$ and let $V\subset U_m$ be an open neighborhood of $\beta$ in $Y$; we will prove
	that there is a periodic point $(0,\alpha)$ with $\mathcal{G}(0,\alpha)\cap \{m\}\times V\neq \emptyset$. 
	Since the finite $G$-orbits are dense, there is some periodic $\alpha\in Y$ such that $\alpha\neq \mu$ and $G\alpha\cap V\neq \emptyset$, so $\mathcal{G}(0,\alpha)\cap\{m\}\times V\neq \emptyset$ by Lemma~\ref{l:gna}. Let us show that $(0,\alpha)$ has a finite $\mathcal{G}$-orbit. Indeed, there is some $k$ such that $G\alpha\notin U_k$, so the set 
	\[
	\{(m,\sigma^n(\alpha))\mid \sigma^n(\alpha)\in U_n\}
	\] 
	is finite, and therefore $(0,\alpha)$ is a periodic point by Lemma~\ref{l:gna}.
\end{proof}

\begin{lemma}\label{l:zeroezroinf}
	$(0,\mu)$ is a point of equicontinuity for $\mathcal{G}$.
\end{lemma}
\begin{proof}
	Let $S$ be the generating set $\{f,f^{-1},g,g^{-1}\}$, then $(0,\mu)\notin \dom h$ for $h\in S$, and the result follows.
\end{proof}

\section{Foliated dynamics}\label{s:foliated}

\subsection{A non-chaotic foliated space with chaotic holonomy pseudogroup.}\label{s:chaoschaos}
Using a construction inspired by Example~\ref{e:rz}, we will show that density of periodic orbits in the holonomy pseudogroup does not imply density of compact leaves. Chaos for foliated spaces is, therefore,  a stronger condition than chaos for pseudogroups, at least with our definitions.

Think of $\mathbb{T}^2$ as the quotient $\mathbb{R}^2/\mathbb{Z}^2$ and consider Arnold's cat map $f\colon \mathbb{T}^2\to \mathbb{T}^2$, which is obtained by factoring the linear map 
\[
\tilde f\colon\mathbb{R}^2\to\mathbb{R}^2,\qquad\tilde f(x,y)=(2x+y,x+y)
\]
through the quotient  $\pi \colon \mathbb{R}^2\to \mathbb{T}^2$.

The cat map $f$ is well-known to be chaotic, so, by Theorem~\ref{t:sicactionpseudogroup}, the pseudogroup $\mathcal{G}\acts\mathbb{T}^2$ generated by $f$ is chaotic too. It is also easy to check that the suspension foliation\footnote{See~\cite[{Section~3.1}]{CandelConlon2000-I}  for an introduction to suspension folations.} induced by the representation $\pi_1(\mathbb{S}^1)\to \Homeo(\mathbb{T}^2)$ sending a generator to $f$ satisfies Definition~\ref{d:chaosfs} and is, therefore, chaotic. 

Consider now the pseudogroup $\mathcal{H}\acts \mathbb{R}^2$ generated by $\tilde f$ and the integer translations. As in Example~\ref{e:rz},	
\[
\Phi:=\{\,\pi|_{U}\mid U\subset \mathbb{R}^2\ \text{open},\ \pi|_{U}\colon U\to \phi(U)\ \text{is a homeomorphism}\,\}
\] 
is an equivalence $\Phi\colon (\mathbb{R}^2,\mathcal{H})\to (\mathbb{T}^2,\mathcal{G})$; $\mathcal{H}$ is then a chaotic pseudogroup by Theorem~\ref{t:invariant}. Let $S$ denote the closed surface of genus three, whose fundamental group has presentation
\[
\pi_1(S)=\langle\, a_1,b_1, a_2,b_2,a_3,b_3\mid [a_1,b_1][a_2,b_2][a_3,b_3]\,\rangle,
\] 
and consider the suspension foliation $X$  induced by the representation $\phi\colon\pi_1(S)\to \Homeo(R^2)$ defined by
\begin{align*}
\phi(a_1)&=\tilde f,& \phi(a_2)&=[(x,y)\mapsto (x+1,y)],\\ \phi(a_3)&=[(x,y)\mapsto (x,y+1)],& \phi(b_1)&=\phi(b_2)=\phi(b_3)=\id.
\end{align*}
The holonomy pseudogroup is obviously equivalent to $\mathcal{H}$, hence chaotic. The foliated space does not have any compact leaves, however, because all $\mathcal{H}$-orbits are infinite, so $X$ is not a chaotic foliated space according to Definition~\ref{d:chaosfs}.

\subsection{A non-compact  topologically transitive foliated space with density of compact leaves which is not sensitive}
\label{s:ttdposicfol}
This section shows that compactness cannot be omitted in the statement of Theorem~\ref{t:ttdposicfol}.
	
Let us use the pseudogroup $\mathcal{G}\curvearrowright X$ from Section~\ref{s:cantor} to obtain a topologically transitive foliated space with density of compact leaves, but with an equicontinuous leaf. We start by constructing a directed graph $Z$, which can be defined as a pair $Z=(V(Z),E(Z))$ where $V(Z)$ is the vertex set and $E(Z)\subset V(Z)\times V(Z)$ is the set of directed edges. We think of $z$ as the origin and $z'$ as the end vertex of the directed edge $(z,z')\in E(Z)$.

Let $V(Z)=X$, and let $E(Z)$ consist of all edges of the form 
$((m,a),f(m,a))$ and $((n,b),g(n,b))$ for $(m,a)\in\dom f$, $(n,b)\in \dom g$.
Let $Y$ denote the set $\{\dom f, \im f, \dom g, \im g\}$, and define the map
 \[\nu\colon X\to 2^{Y}, \quad \nu(x)(A)=\begin{cases}1 \quad \text{ if } x\in A,\\ 0\quad \text{ else,}\end{cases}\]where $A\in Y$. Under the usual identification of maps in $2^Y$ and subsets of $Y$, $\nu(x)$ is the subset of $Y$ containing exactly the elements of $Y$  that contain $x$.
Since $\dom f$, $\im f$, $\dom g$, and $\im g$ are clopen, the sets $\nu^{-1}(C)$ for $C\subset Y$ form a partition of $Y$ by clopen sets. Choose disjoint open balls in the two-sphere $\mathbb{S}^2$ indexed by the elements $A\in Y$, and denote them by $B_A$. 

For $C\subset Y$, let $S_C:=\mathbb{S}^2\setminus \bigcup_{A\in C} B_A$, and denote the boundary circles in $S_C$ by $\Delta_A$ ($A\in C$). Let 
\begin{align*}
	\mathfrak{Y}_1=\bigsqcup_{C\subset Y}\nu^{-1}(C)\times S_C,\qquad  \mathfrak{Y}_2=E(Z)\times \mathbb{S}\times[0,1].
\end{align*}

Assume that we have fixed identifications of the boundary circles $\Delta_A$ with $\mathbb{S}$.
Let $\mathfrak{X}$ be the following quotient of $\mathfrak{Y}_1\sqcup \mathfrak{Y}_2$: for each $((m,a),h(m,a))\in E(Z)$ with $h\in\{f,g\}$,  identify
\begin{align*}
	\{((m,a),h(m,a))\}\times \mathbb{S}\times\{0\}\ &\sim  \ \{((m,a),h(m,a))\}\times \Delta_{\dom h},\\ \{((m,a),h(m,a))\}\times \mathbb{S}\times\{1\} \ &\sim \ \{((m,a),h(m,a))\}\times \Delta_{\im h}.
\end{align*}

First, note that, locally, the space looks like the product of $\mathbb{R}^2$ times a Cantor set, and is therefore a matchbox manifold. Take any point $u\in \mathbb{S}^2$ that is not contained in any $B_A$, $A\in Y$; the image by the quotient map $\pi\colon\mathfrak{Y}_1\sqcup\mathfrak{Y}_2\to\mathfrak{X}$ of \[\bigsqcup_{C\subset Y}\nu^{-1}(C)\times\{u\}\cong X\times \{u\}\subset \mathfrak{Y}_1\]
meets every connected component of $\mathfrak{X}$, and is therefore a transversal meeting every leaf. Every path in $\mathfrak{X}$ is homotopic to the concatenation of paths of the form $\pi(\tau)$, where $\tau$ is either a path contained in a plaque $\{(n,a)\}\times S_C\subset \mathfrak{Y}_1$ for $(n,a)\in \nu^{-1}(C)$, $C\subset Y$, or the path $t\in[0,1]\mapsto (((n,a),h(n,a)),s,\pm t)$ for $((n,a),h(n,a))\in E(Z)$ and $s$ a point in a boundary circle. For $\tau$ of the first type, $\pi(\tau)$ is contained in the same plaque, so it induces an identity transformation on the transversal; for $\tau$ of the second kind, the path begins in the plaque corresponding to the point $(u,(n,a))$ and ends at the point $(u,h^{\pm1}(n,a))$, where $h$ is either $f$ or $g$. So these paths induce $f$, $f^{-1}$, $g$ or $g^{-1}$ on the transversal, and we see that the holonomy pseudogroup is equivalent to $\mathcal{G}\curvearrowright X$. 

We conclude that $\mathfrak{X}$ is a topologically transitive matchbox manifold and the leaf corresponding to $(0,0^\infty)\in X$ is a leaf of equicontinuity by Lemma~\ref{l:zeroezroinf}. Moreover, the periodic points of the action on $X$ correspond to compact leaves in $\mathfrak{X}$, so we also have density of compact leaves. 

\subsection{An affine pseudogroup}\label{s:affinepseudogroup}
Before embarking on the proof of Theorem~\ref{t:counterfol}, we need to obtain a modified version of the pseudogroup in Section~\ref{s:nonsensitiveaction}. The reason is that we will construct the foliated space counterexample with a particular representative of the holonomy pseudogroup in mind, but we cannot use the pseudogroup of Section~\ref{s:nonsensitiveaction} because all its orbits are infinite.

We start by fixing the following notation
\begin{equation}
\label{defnlr}	l_n^-=\frac{1}{3\cdot2^{1+n}},\quad r_n^-=\frac{1}{3\cdot2^{n}},\quad l_n^+=1-r_n^-,\quad r_n^+=1-l_n^-,
\end{equation}
and then we fix the following intervals in order to define a sequence of toral linked twists:
\begin{align}
	H&=\{(x,y)\in \mathbb{T}^2\mid 1/6\leq y\leq 5/6\},\notag\\
	V_0&=\left\{(x,y)\in \mathbb{T}^2\mid 1/6\leq x\leq 5/6\right\},\notag\\
	V_n^-&=\left\{(x,y)\in \mathbb{T}^2\mid l_n^-\leq x\leq r_n^-\right\}\qquad \text{for}\ n\geq 1,\notag\\
	V_n^+&=\left\{(x,y)\in \mathbb{T}^2\mid l_n^+\leq x\leq r_n^+\right\} \qquad \text{for}\ n\geq 1.\notag
\end{align}
Let $T_h\colon \mathbb{T}^2\to \mathbb{T}^2$ be the horizontal twist defined by
\begin{equation}\label{defnth}
T_h(x,y)=\begin{cases}
	(x+6(y-\frac{1}{6}),y)\quad &\text{if}\ (x,y)\in H,\\
	(x,y)\quad &\text{else};
\end{cases}
\end{equation}
and, for $m\in\mathbb{N}$, let $T_{v,m}\colon\mathbb{T}^2\to \mathbb{T}^2$ be  the vertical twist:
\begin{equation}\label{defntvm}
T_{v,m}(x,y)=\begin{cases}
	(x,y+6(x-1/6))\quad &\text{if}\ (x,y)\in V_0,\\
	(x,y+3\cdot 2^{1+n}(x-l_n^-))\quad &\text{if}\ (x,y)\in V_n^-,\ n\leq m,\\
	(x,y+3\cdot 2^{1+n}(x-l_n^+))\quad &\text{if}\ (x,y)\in V_n^+,\ n\leq m,\\
	(x,y)\quad &\text{else}.
\end{cases}
\end{equation}
Finally, let $T_m=T_{v,m}\circ T_h$ be the corresponding linked twist map.

We denote by $\Delta_n$ the set of points in $\mathbb{T}^2$ where $T_n$ is not smooth; i.e.,
\[
\Delta_n= \partial H\cup \bigcup_{m\leq n}T_h^{-1}(\partial V^-_m)\cup \bigcup_{m\leq n}T_h^{-1}(\partial V^+_m).
\]
Finally, let 
\begin{align*}
\Delta&=\bigcup_n \Delta_n,\\
M_n&:=H\cup \bigcup_{m\leq n}V_m^+\cup \bigcup_{m\leq n}V_m^-,
\end{align*}
and note that $\Delta_n\subset M_n$ and $M_n$ is the set where  $T_m$ is topologically transitive and sensitive to initial conditions by Theorem~\ref{t:twisttt}.

The linear nature of these linked twists will allow us to find a common set of periodic orbits. Let 
\[
Q_n= M_n\cap \left\{\left(\frac{l_1}{2^{n}},\frac{l_2}{2^{n}}\right)\mid l_1,l_2\in \{0,\ldots, 2^{n}-1\}\right\}, \quad n\geq 1.
\]

\begin{lemma}\label{l:tq}
	$T_m(Q_n)=Q_n$ for every $n$ and $m$.
\end{lemma}
\begin{proof}
	We can rewrite~\eqref{defnth} and, using~\eqref{defnlr},  also rewrite~\eqref{defntvm} as
	\begin{align*}
		T_h(x,y)&=\begin{cases}
			(x+6y,y)\quad &\text{if}\ (x,y)\in H,\\
			(x,y)\quad &\text{else};
			\end{cases}\\
			T_{v,m}(x,y)&=\begin{cases}
				(x,y+6x)\quad &\text{if}\ (x,y)\in V_0,\\
				(x,y+3\cdot 2^{1+n}x)\quad &\text{if}\ (x,y)\in V_n^-,\ n\leq m,\\
				(x,y+3\cdot 2^{1+n}x)\quad &\text{if}\ (x,y)\in V_n^+,\ n\leq m,\\
				(x,y)\quad &\text{else}.
			\end{cases}
	\end{align*}
The result is now obvious.
\end{proof}
Note that $Q_n\cap \partial H=\emptyset$ because the points in $\partial H$ have $y$-coordinate $1/6$ or $5/6$, which cannot be expressed as a fraction whose denominator is a power of $2$. Similarly, the above expression for $T_h$ and the definitions of $V_m^+$ and $V_m^-$ show that $Q_n\cap \Delta_m=\emptyset$ for every $n$ and $m$, or, equivalently,
\begin{equation}	\label{qndelta}
	Q_n\cap \Delta =\emptyset\qquad \text{for every} \ n.
\end{equation}
Also, by defining 
\[
\widetilde{Q}_0=Q_0=\emptyset,\qquad
\widetilde{Q}_n=Q_n\setminus Q_{n-1} \quad\text{for}\ n\geq1,
\]
we obtain by Lemma~\ref{l:tq}
\begin{equation}\label{tildeq}
T_m(\widetilde{Q}_n)=\widetilde{Q}_n\qquad \text{for every}\quad n,m.
\end{equation}

Up to this point, we have proceeded almost in the same way as in Section~\ref{s:nonsensitiveaction}. There, we used the pseudogroup on $Y:=\mathbb{T}^2\times \mathbb{N}$ defined by the maps
\[
((x,y),n)\mapsto (T_n(x,y),n)\qquad \text{and} \qquad ((x,y),n)\mapsto ((x,y),n+1).
\]
In order to get affine maps, we will restrict the first map to an appropiate subspace; to obtain density of finite orbits, we will ``cut" open balls with center points in $Q_n$ out of the domain of the second map. Let us start by finding suitable radii.

\begin{lemma}\label{l:radii}
	There is a decreasing sequence $r_0,r_1,\ldots$ of positive radii  such that, for every $n\geq 0$ and $(x,y)\in \widetilde Q_n$,
	\begin{enumerate}[(i)]
		\item
		\label{i:radiimn}$B((x,y),r_n)\subset M_n$,
		\item \label{i:radiidelta}$d(B((x,y),r_n), \Delta_{n})>0$,
		\item \label{i:radiidisjoint}		for every $0\leq m<n$ and $(x',y')\in \widetilde Q_{m}$,
		\[
		d((x,y),(x',y'))>r_m-r_n\quad  \Longrightarrow\quad  d((x,y),(x',y'))>r_m+r_n,
		\] 
		\item \label{i:radiisphere}and $\partial B((x,y),r_n)$ consists of points with at least one irrational coordinate.
	\end{enumerate}
\end{lemma}
\begin{proof}
	We proceed by induction on $n\geq 0$; we will choose all radii $r_n$ to be trascendental numbers. First, we may set $r_0$ arbitrarily because $\widetilde Q_0$ is empty. Assume now that we have chosen $r_0,\ldots,r_{n-1}$ satisfying the above conditions, then (\ref{i:radiimn})--(\ref{i:radiidelta}) hold for $r_n$ small enough because of~\eqref{qndelta}. To prove that~(\ref{i:radiidisjoint}) holds for small enough radii, note that every $\widetilde Q_m$ is a finite set consisting of points with rational coordinates, so $d((x,y),\widetilde Q_m)$ is an algebraic number for every $(x,y)\in \widetilde Q_n$ and $0\leq m <n$. Since the $r_m$ are transcendental numbers,~(\ref{i:radiidisjoint}) is satisfied for small $r_n$. Finally, we can choose $r_n$ transcendental and satisfying (\ref{i:radiisphere}) because there are countably many points with both coordinates rational but uncountably many transcendental radii satisfying (\ref{i:radiimn})--(\ref{i:radiidisjoint}).
\end{proof}

Let 
\begin{align}\label{un}
	\mathcal U_n&:=\bigcup_{(x,y)\in \widetilde Q_n} B((x,y),r_{n}),\\
	\mathcal{V}_n&:=\label{defnutilde} \bigcup_{0\leq m \leq n} \mathcal{U}_n.
\end{align}
By Lemma~\ref{l:radii}(\ref{i:radiidisjoint}), we can express $\mathcal{V}_n$ as a disjoint union of open balls. It follows that  $\mathcal{V}_n\cap M_l$ is not dense in $M_l$ for any $n,l\geq 0$.

We are finally prepared to define what will be the holonomy pseudogroup of our counterexample foliated space. Let $\tilde f$ and $\tilde g$ be maps on $Y$ defined by
\begin{align}\label{defntildef}
	\tilde f((x,y),n)&=(T_n(x,y),n),\quad 	&\dom \tilde f &=\{((x,y),n)\mid (x,y)\notin \Delta_n\},\\ 
	\label{defntildeg}\tilde g((x,y),n)&=((x,y),n+1),\quad 	&\dom \tilde g& =\{((x,y),n)\mid (x,y)\notin\mathcal{V}_n\},
\end{align}
and let $\widetilde{G}\acts Y$ be the pseudogroup generated by $\tilde f$ and $\tilde g$.

\begin{lemma}
	The pseudogroup $\widetilde{\mathcal{G}}\acts Y$  is topologically transitive.
\end{lemma}
\begin{proof}
	Since $T_n$ is topologically transitive on $M_n$ for every $n\geq 0$, there are residual sets $R_n\subset M_n$  consisting of points whose $T_n$-orbits are dense in $M_n$; hence, the fact that $\Delta$ is meager yields that
	\[
	R:=\left(\bigcap_{n\geq 0, z\in\mathbb{Z}} T^{z}_0(R_n) \right)\setminus \left(\bigcup_{n\geq 0, z\in\mathbb{Z}} T^{z}_n(\Delta) \right)
	\] 
	is a residual subset of $M_0$ satisfying 
	\begin{equation}\label{deltar}
		\Delta\cap \bigcup_{n\geq 0,z\in\mathbb{Z}} T_n^z(R)=\emptyset,
	\end{equation}
	\begin{equation}\label{r0r}
		(x,y)\in R_0\qquad \Longrightarrow\qquad  T_0^z(x,y)\in R_n \qquad \text{for every}\ n\geq 0, z\in\mathbb{Z}.
	\end{equation}

	We will prove that any point $((x,y),0)\in Y$  with $(x,y)\in R$ has a dense $\widetilde G$-orbit. Consider an open set $V\times\{n\}$ with $V\subset \mathbb{T}^2$ and $n\geq 0$, then there is a least $m\geq n$ such that $V\cap M_m\neq \emptyset$. Since $ \mathcal{V}_l\cap M_0$ is not dense in $M_0$ for any $l\in \mathbb{N}$ and $T_0(x,y)$ is dense in $M_0$, there is $z_1\in \mathbb{Z}$ such that 
	\[
	T_0^{z_1}(x,y)\notin \mathcal{V}_l\qquad \text{for}\ l=0,\ldots,m.
	\]
	This means that $((x,y),0)\in\dom\tilde g^m$ by~\eqref{defntildeg}.
	Equations~\eqref{deltar} and~\eqref{r0r}
	now yield
	\[
	T_0^{z_1}(x,y)\in R_m \setminus \bigcup_{z\in \mathbb{Z}}T_m^{z}(\Delta),
	\]
	and therefore
	\[
	\tilde g^m\tilde f^{z_1}(x,y)\in (R_m \setminus \bigcup_{z\in \mathbb{Z}}T_m^{z}(\Delta))\times \{m\}.
	\]
	By the definition of $R_m$, the orbit
	\[\bigcup_{z\in\mathbb{Z}}T_m^z(T_0^{z_1}(x,y))\] is dense in $M_m$ and disjoint from $\Delta$, so, by~\eqref{defntildef}, there is $z_2\in\mathbb{Z}$ such that 
	\[
	(x',y'):=\tilde f^{z_2}\tilde g^m\tilde f^{z_1}(x,y)\in (V\cap M_m)\times \{m\}.
	\] 
	
	We defined $m$ as the least integer $\geq n$ satisfying $V\cap M_m\neq \emptyset$, so we have $V\cap M_l=\emptyset$ for every $n\leq l<m$. Hence $(x',y')\in \dom (\tilde g^{n-m})$ by~\eqref{defntildeg} and Lemma~\ref{l:radii}(\ref{i:radiimn}), yielding
	\[
	\tilde g^{n-m}\tilde f^{z_2}\tilde g^m\tilde f^{z_1}(x,y)\in V\times \{n\}.
	\]
	We have proved that, for $(x,y)\in R$, the $\widetilde{\mathcal{G}}$-orbit of $((x,y),0)$ meets every open set, so $\widetilde{\mathcal{G}}\acts Y$ is topologically transitive by Lemma~\ref{l:pt}.
\end{proof}

\begin{lemma}\label{l:finiteorbits}
	Let $((x,y),m)\in \widetilde Q_n\times\{m\}$ with $m\leq n$, then the orbit $\widetilde{\mathcal{G}}((x,y),m)$ is contained in $ \widetilde Q_n\times\{0,\ldots,n\}$. 
\end{lemma}
\begin{proof}
	We have $T_l(\widetilde Q_n)=\widetilde Q_n$ for every $0\leq l\leq n$ by~\eqref{tildeq}, so 
	\[
	\tilde f(\bigcup_{m\leq n} \widetilde Q_n\times\{m\})\subset \bigcup_{m\leq n} \widetilde Q_n\times\{m\},
	\]
	and similarly for $\tilde f^{-1}$. It is obvious from the definition of $\tilde g$ that
	\[
	\tilde g^{-1}(\bigcup_{m\leq n}\widetilde Q_n\times\{m\})\subset \bigcup_{m\leq n}\widetilde Q_n\times\{m\}, \quad \tilde g(\widetilde Q_n\times\{m\})\subset \bigcup_{m\leq n}\widetilde Q_n\times\{m+1\}.
	\]
	Hence, to finish the proof it is enough to show that $(\widetilde Q_n\times\{n\})\cap \dom \tilde g=\emptyset$, but this follows from~\eqref{un},~\eqref{defnutilde} and~\eqref{defntildeg}.
\end{proof}

\begin{corollary}
	The finite $\widetilde{\mathcal{G}}$-orbits are dense in $Y$.
\end{corollary}

The proof of the following statement is identical to that of Proposition~\ref{p:gmathcalg}.

\begin{proposition}\label{p:widetildegnotsic}
	There is a metric $d$ on $Y$ and a generating pseudo{\textasteriskcentered}group $\mathcal{S}$ such that every map in $\mathcal{S}$ whose domain contains $((0,0),0)$ is an isometry. Hence, $\widetilde{\mathcal{G}}$ is not sensitive to initial conditions.
\end{proposition}

\subsection{A non-compact, topologically transitive affine foliation with a dense set of compact leaves which is not sensitive}
\label{s:ttdclnotsicaffine}
We are now in position to prove Theorem~\ref{t:counterfol} by constructing a suitable foliated space using the pseudogroup we have just defined.

Let $\Sigma$ be a smooth surface of genus two  divided into three smooth manifolds with boundary, $\Sigma_0$, $\Sigma_\alpha$, and $\Sigma_\beta$,
that overlap only on their boundaries. Let $\Sigma_0$ be a two-sphere with four open disks removed, and denote the  boundary circles by $S_\alpha^-$, $S_\alpha^+$, $S_\beta^-$, $S_\beta^+$, with $\Sigma_\alpha$ attaching at $S_\alpha^-$ and $S_\alpha^+$ (see Figure~\ref{f:doubletorus}).

\begin{figure}[tbh]
	\includegraphics[width=0.6\textwidth]{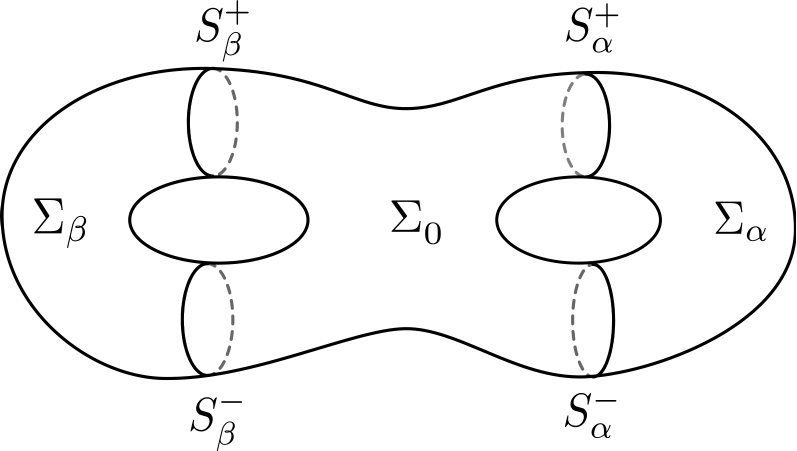}
	\caption{The surface $\Sigma$ and its partition}
	\label{f:doubletorus}
\end{figure}

The main idea of the construction is as follows: We have constructed a pseudogroup $\widetilde{\mathcal{G}}\acts Y$ generated by two maps $\tilde f$ and $\tilde g$  satisfying suitable properties. In order to get a foliated manifold realizing this dynamics, the obvious candidate is to mimic a suspension foliation by taking products $\Sigma_\alpha\times \dom \tilde f$, $\Sigma_\beta\times \dom \tilde g$ and $\Sigma_0\times Y$, and attaching each plaque in $\Sigma_\alpha\times \dom \tilde f$ so that we identify $(x,y)\in\Sigma_\alpha\times \dom \tilde f$ and $(x,y)\in\Sigma_0\times \dom \tilde f$ if $x\in S_\alpha^-$ and  we identify $(x,y)\in\Sigma_\alpha\times \dom \tilde f$ and $(x,\tilde f(y))\in\Sigma_0\times \im \tilde f$ if $x\in S_\alpha^+$. In this way, the plaques of $\Sigma_\alpha\times \dom \tilde f$ realize $\tilde f$ in the holonomy pseudogroup. Proceeding similarly, we realize $\tilde g $ using $\Sigma_\beta\times \dom \tilde g$. However, now there are singularities at the points \[(x,y)\in (S_\alpha^-\times \partial\dom \tilde f)\sqcup (S_\alpha^+\times \partial\im \tilde f),\] and similarly for $\tilde g$, so the resulting space would not even be a manifold with boundary. Since we do not require a compact foliated manifold, we may get rid of the problematic points and consider $\mathfrak Y_0:= (\Sigma_0\times Y)\setminus \mathfrak{C}$, where 
\begin{align*}
	\mathfrak{C}&=\mathfrak{C}_\alpha^-\cup \mathfrak{C}_\alpha^+\cup \mathfrak{C}_\beta^-\cup \mathfrak{C}_\beta^+,\\
	\mathfrak{C}_\alpha^-&= S_\alpha^-\times \partial\dom \tilde f,\\
	\mathfrak{C}_\alpha^+&= S_\alpha^+\times \partial \im \tilde f,\\
	\mathfrak{C}_\beta^-&= S_\beta^-\times \partial \dom \tilde g,\\
		\mathfrak{C}_\beta^+&= S_\beta^+\times \partial \im \tilde g.
\end{align*}

Recall that $\dom \tilde f$ is a dense open subset of $Y$ and $\dom \tilde g$ is the complement of a disjoint union of open balls. Now $\mathfrak Y_0$ is a non-compact smooth manifold with boundary. 

Now attach the boundaries of $\mathfrak Y_\alpha:=\Sigma_\alpha\times \dom \tilde f$ and $\mathfrak Y_\beta:=\Sigma_\beta\times \dom \tilde g$ to $\mathfrak Y_0$ using the identifications
\begin{align*}
(s,y)&\sim(s,y),\quad &(s,y)\in  S_\alpha^-\times\dom \tilde f,\\
(s,y)&\sim (s,\tilde f(y)),\quad&(s,y)\in  S_\alpha^+\times\dom \tilde f,\\
(s,y)&\sim (s,y),\quad&(s,y)\in  S_\beta^-\times\dom \tilde g,\\
(s,y)&\sim (s,\tilde g(y)),\quad&(s,y)\in  S_\beta^+\times\dom \tilde g.\\
\end{align*}

Denote by $\mathfrak{Y}$ the resulting space. The product foliated structures on $\mathfrak Y_0$, $\mathfrak Y_\alpha$, and $\mathfrak{Y}_\beta$ with leaves $\{y\}\times \Sigma_i$ (where $i=0$, $\alpha$, or $\beta$, respectively) descend to  $\mathfrak{Y}$. 
It is an elementary matter to check that, essentially by construction, $\mathfrak{Y}$ is $C^\infty$ and its holonomy pseudogroup is equivalent to $\widetilde{\mathcal{G}}\acts Y$. Indeed, we can choose a point $x$ in the interior of $\Sigma_0$, and then $\{x_0\}\times Y$ is a total transversal of $\mathfrak{Y}$. Since $\mathfrak{Y}_0$ has a product foliation structure, it has trivial dynamics; by gluing $\mathfrak{Y}_\alpha$ and $\mathfrak{Y}_\beta$, we realize $\tilde f$ and $\tilde g$, respectively, so the holonomy pseudogroup is equivalent to $\widetilde{\mathcal{G}}\acts Y$.

\begin{corollary}\label{c:mathfraky}
	The foliated space $\mathfrak{Y}$ is $C^\infty$, transversally affine and  topologically transitive, but not sensitive to initial conditions.
\end{corollary}

The space $\mathfrak{Y}$ is also a smooth manifold with boundary and, noticing that we removed $\mathfrak{C}$ from $(\Sigma_0\times Y)\setminus \mathfrak{C}$ and keeping track of which points in the boundary of $\mathfrak{Y}_0$ we glued to $\mathfrak{Y}_\alpha$ and $\mathfrak{Y}_\beta$, we see that the boundary of $\mathfrak{Y}$ is
\begin{equation}\label{eq:Yboundary}
\partial \mathfrak{Y} =   S_\beta^-\times(Y\setminus \overline{\dom \tilde g}) \sqcup  S_\beta^+\times(Y\setminus \overline{\im \tilde g}).
\end{equation}
Recall that $\dom \tilde f$ is an open and dense subset so the terms containing $(Y\setminus \overline{\dom \tilde g})$ are empty; the same argument applies to $\im \tilde f$.

\begin{lemma}\label{l:mathfrakydpo}
	The foliated space
	$\mathfrak{Y}$ has a dense set of leaves that are compact manifolds, possibly with boundary.
\end{lemma}
\begin{proof}
	By Lemma~\ref{l:finiteorbits}, a leaf $L$ corresponding to a point $((x,y),m)$ with $(x,y)\in Q_n$ corresponds to a finite orbit of $\widetilde{\mathcal G}$. Consider the inverse image of $L$ by the quotient map $\mathfrak{Y}_0\sqcup ( \dom \tilde f\times\Sigma_\alpha) \sqcup (\dom \tilde g\times\Sigma_\beta)\to \mathfrak{Y}$.  	Since it has a finite orbit, the inverse image of $L$ only intersects finitely many plaques in $\mathfrak{Y}_0$, $\Sigma_\alpha \times \dom \tilde f$, and $\Sigma_\beta\times\dom \tilde g$. The plaques in $\Sigma_\alpha \times \dom \tilde f$, and $\Sigma_\beta\times\dom \tilde g$ are all compact, being copies of $\Sigma_\alpha$ and $\Sigma_\beta$, respectively. By Lemma~\ref{l:radii} and~\eqref{un}--\eqref{defntildeg}, 
	\[
	Q_n\cap \bigcup_{l\geq 0}\partial \mathcal{V}_l=Q_n\cap \Delta=\emptyset,
	\] 
	so every $\widetilde{\mathcal{G}}$-orbit of a point $((x,y),m)$ is disjoint from 
 \[\partial\dom \tilde f\cup\partial\im\tilde f\sqcup\partial\dom \tilde g\cup\partial\im\tilde g.\]
	Looking back at the construction of $\mathfrak{Y}_0$ at the beginning of this section, this implies that the plaques in $\mathfrak{Y}_0$ that project to $L$ are all compact, being copies of $\Sigma_0$; $L$ is then the quotient of a finite union of compact plaques, hence compact. 
\end{proof}

At this point, we have constructed a smooth and transitive foliated manifold with boundary $\mathfrak{Y}$ that has a dense set of compact leaves but is not sensitive to initial conditions. It only remains to modify it in order to get rid of the boundary: Take two copies $\mathfrak{Y}^-$ and $\mathfrak{Y}^+$, and let $\mathfrak{Y}^\pm\cong \mathfrak{Y}^-\cup\mathfrak{Y}^+/\sim$ be the quotient space obtained by identifying their boundaries. This is sometimes called the ``double'' of a manifold with boundary, and it is known to admit a smooth structure making it a manifold without boundary. In our case, where we constructed our space $\mathfrak{Y}$ by gluing product manifolds endowed with  product foliations, it is elementary to check that the foliated structure descends to the quotient and $\mathfrak{Y}^\pm$ is now a smooth foliated manifold without boundary.

\begin{lemma}
	The set of compact leaves is dense in $\mathfrak{Y}^\pm$.
\end{lemma}
\begin{proof}
Denote by $\pi\colon \mathfrak{Y}^- \sqcup \mathfrak{Y}^+ \to \mathfrak{Y}^\pm$ the quotient map.
Let $U$ be a open subset of $\mathfrak{Y}^\pm$, which without loss of generality we may assume contained in $\pi(\mathfrak{Y}^-\setminus \partial \mathfrak{Y}^-)$. By Lemma~\ref{l:mathfrakydpo}, there is a compact leaf $L^-$ in $\mathfrak{Y}^-$ intersecting $\pi^{-1}(U)$. If $L^-$ has empty boundary, then $L^-\cap \partial \mathfrak{Y}^- =\emptyset$ and $\pi(L^-)$ is  a compact leaf in the quotient space $\mathfrak{Y}^\pm$ intersecting $U$. If $L^-$ has non-empty boundary, then it is contained in the leaf $L=\pi(L^-\sqcup L^+)$, where $L^+$ is the leaf of $\mathfrak{Y}^+$ that corresponds to $L^-$. Then $L$ intersects $U$ and is a quotient of the compact space $L^-\sqcup L^+$, hence compact.
\end{proof}

\begin{lemma}
	$\mathfrak{Y}^\pm$ is topologically transitive but not sensitive to initial conditions.
\end{lemma}
\begin{proof}
    Let's inspect the holonomy pseudogroup of $\mathfrak{Y}^\pm$. Since $Y\cong \{x_0\}\times Y$ was a total transversal for $\mathfrak{Y}$, for $\mathfrak{Y}\pm$ we can take $Y^\pm \equiv Y^-\cup Y^+$, where $Y^-$ and $Y^+$ denote the image by $\pi$ of copies of $\{x_0\}$ in $\mathfrak{Y}^-$ and $\mathfrak{Y}^+$, respectively. On each of $Y^-$ and $Y^+$ we have copies of $\widetilde{\mathcal{G}}$, denoted by $\widetilde{\mathcal{G}}^-$ and $\widetilde{\mathcal{G}}^+$. The maps between $Y^-$ and $Y^+$ come from the gluing of the boundaries of $\mathfrak{Y}^-$ and $\mathfrak{Y}^+$. Recall that $Y= \mathbb{T}^2\times\mathbb{N}$. Loooking at~\eqref{eq:Yboundary}, we see that, if we denote the points of  $Y^-$ and  $Y^+$ by $(y,-)$ and $(y,+)$, respectively, we get a new map $h$ defined by
    \begin{align*}
    \dom h &= \{(y,-)\mid y \in Y^-\setminus (\overline{\dom \tilde g^-} \cap \overline{\im \tilde g^-})\},\\
    \im h &= \{(y,+)\mid y \in Y^+\setminus (\overline{\dom \tilde g^+} \cap \overline{\im \tilde g^+})\},\\
    h(y,-)&=(y,+),
    \end{align*}
    where  $\tilde g^-$ and $\tilde g^+$ denote the respective copies of  $\tilde g$ acting on $Y^-$ and $Y^+$.

    Denote by $\widetilde{G}^\pm$ the pseudogroup acting on $Y^\pm$ generated by $\widetilde{G}^-$, $\widetilde{G}^+$, and $h$, which is then the holonomy pseudogroup of $\mathfrak{Y}^\pm$.
	Let us prove that is topologically transitive. By Corollary~\ref{c:mathfraky}, there is a dense orbit $\widetilde{\mathfrak{G}}(y)$ in $Y$. Then the orbit $\widetilde{\mathfrak{G}}^\pm(y,-)$ contains $\widetilde{\mathfrak{G}}^-(y,-)$, which is dense in $Y^-$, and similarly $\widetilde{\mathfrak{G}}^\pm(y,+)$ is dense in $Y^+$. Since $\dom \tilde g$ was a disjoint union of open balls in $Y$,  $\widetilde{\mathfrak{G}}^-(y,-)$ meets $\dom h$, and we get \[\widetilde{\mathfrak{G}}^\pm(y,-) = \widetilde{\mathfrak{G}}^-(y,-)\cup \widetilde{\mathfrak{G}}^+(y,+),\]
	which is dense in $Y^\pm$.
	
	Finally, let us prove that $\widetilde{\mathcal{G}}^\pm$ is not sensitive to initial conditions. By Proposition~\ref{p:widetildegnotsic}, there is a metric $d$ on $Y$, a point $y\in Y$ and a generating pseudo{\textasteriskcentered}group $S$ for $\widetilde{\mathcal G}$ such that every map of $S$ whose domain contains $y$ is an isometry. 
	Let $d^\pm$ be metric on $Y^\pm$ such that points from $Y^-$ and $Y^+$ are at infinite distance of each other, and the restrictions of $d^\pm$ to $Y^-$ and $Y^+$ coincide with $d$. Finally, let $S^\pm=S^-\cup S^+\cup \{ h\}$, where $S^-$ and $S^+$ are copies of $S$ acting on $Y^-$ and $Y^+$, respectively. It is immediate that $S^\pm$ generates $\widetilde{\mathcal{G}}^\pm$ and  every map in $S^\pm$ whose domain contains $(y,-)$ is an isometry with respect to $d^\pm$, and the result follows.
\end{proof}

The pseudogroup $\mathcal{G}\acts Y$ was affine, and it is easily checked that so is $\widetilde{\mathcal G}^\pm\acts Y^\pm$, yielding that $\mathfrak{Y}^\pm$ is a transversally affine foliation. This completes the proof of Theorem~\ref{t:counterfol}.

\subsection*{Acknowledgements}
I would like to thank the anonymous referee for their thoughtful suggestions, which helped improved the paper.

\end{document}